\documentclass[11pt]{article}
\usepackage[T1]{fontenc}
\usepackage{lmodern}
\usepackage{amsmath,amsthm,amssymb,amsfonts}
\usepackage{enumerate}
\usepackage{epsfig}
\usepackage[dvipsnames,svgnames,table]{xcolor}
\usepackage{pgf,tikz}
\usetikzlibrary{calc}
\usepackage{fullpage}
\usepackage[colorlinks=true,linkcolor=RoyalBlue,urlcolor=RoyalBlue,citecolor=red]{hyperref}
\usepackage{booktabs}
\usepackage[nameinlink]{cleveref}
\usepackage{enumitem}
\usepackage{optidef}

\newtheorem{theorem}{Theorem}[section]
\newtheorem{proposition}[theorem]{Proposition}
\newtheorem{lemma}[theorem]{Lemma}
\newtheorem{corollary}[theorem]{Corollary}
\newtheorem{conjecture}[theorem]{Conjecture}
\newtheorem{claim}[theorem]{Claim}

\theoremstyle{definition}
\newtheorem{definition}[theorem]{Definition}
\newtheorem{example}[theorem]{Example}

\theoremstyle{remark}
\newtheorem{remark}[theorem]{Remark}

\let\emph\relax % there's no \RedeclareTextFontCommand
\DeclareTextFontCommand{\emph}{\bfseries\em}

\DeclareMathOperator{\rank}{rank}
\DeclareMathOperator{\col}{col}
\DeclareMathOperator{\sgn}{sgn}
\DeclareMathOperator{\smooth}{smooth}
\DeclareMathOperator{\Isom}{Isom}

\definecolor{colR}{rgb}{.932,.172,.172}
\definecolor{colB}{rgb}{.255,.41,.884}
\definecolor{colG}{rgb}{0,0.7,0}

\tikzstyle{edge}=[line width=1.5pt]
\tikzstyle{dedge}=[edge,dashed,gray]
\tikzstyle{redge}=[edge,colR]
\tikzstyle{bedge}=[edge,colB]
\tikzstyle{gedge}=[edge,gray]
\tikzstyle{lnode}=[circle,white,draw, fill=black,inner sep=1pt, font=\scriptsize]
\tikzstyle{hollow}=[circle,gray,draw, thick, fill=gray,inner sep=0pt, minimum size=4pt]
\tikzstyle{vertex}=[fill=black,circle,inner sep=0pt, minimum size=4pt]

\title{Uniquely realisable graphs in polyhedral normed spaces}
\author{Sean Dewar\thanks{School of Mathematics, University of Bristol. E-mail: \texttt{sean.dewar@bristol.ac.uk}}
}

\begin{document}
\date{}
\maketitle

\begin{abstract}
	A framework (a straight-line embedding of a graph into a normed space allowing edges to cross) is globally rigid if any other framework with the same edge lengths with respect to the chosen norm is an isometric copy.
	We investigate global rigidity in polyhedral normed spaces: normed spaces where the unit ball is a polytope.
	We first provide a deterministic algorithm for checking whether or not a framework in a polyhedral normed space is globally rigid.
	After showing that determining if a framework is globally rigid is NP-Hard, we then provide necessary conditions for global rigidity for generic frameworks.
	We obtain stronger results for generic frameworks in $\ell_\infty^d$ (the vector space $\mathbb{R}^d$ equipped with the $\ell_\infty$ metric) 
	including an exact characterisation of global rigidity when $d=2$, and an easily-computable sufficient condition for global rigidity using edge colourings.
	Our 2-dimensional characterisation also has a surprising consequence: Hendrickson's global rigidity condition fails for generic frameworks in $\ell_\infty^2$.
\end{abstract}



\section{Introduction}\label{sec:intro}

A \emph{framework} in a normed space $(\mathbb{R}^d,\|\cdot\|)$ is a pair $(G,p)$ where $G=(V,E)$ is a (finite simple) graph and $p:V \rightarrow \mathbb{R}^d$ is a map called a \emph{realisation} of $G$. 
We say the framework $(G,p)$ is \emph{globally rigid} if any other framework in $(\mathbb{R}^d,\|\cdot\|)$ with the same edge lengths (with respect to the norm $\|\cdot\|$) is an isometric copy of $(G,p)$. 
We can alternatively define this framework property using \emph{rigidity map}
\begin{equation*}
	f_G : (\mathbb{R}^d)^V \rightarrow \mathbb{R}^E, ~ p \mapsto \big( \|p(v) - p(w)\| \big)_{vw \in E}.
\end{equation*}
With this, a framework $(G,p)$ in $(\mathbb{R}^d,\|\cdot\|)$ is globally rigid if and only if for any other realisation $q \in (\mathbb{R}^d)^V$ satisfying $f_G(q) = f_G(p)$, there exists an isometry $h : \mathbb{R}^d \rightarrow \mathbb{R}^d$ such that $h \circ q = p$.
A slight weakening of global rigidity is to only require ``sufficiently close'' edge-length equivalent frameworks to be isometric copies.
Any framework satisfying this weaker property is said to be \emph{rigid}.
Both rigidity \cite{Dewar20,Kitson15,KP14,petty} and global rigidity \cite{DHN22,st24} have previously been studied in the general setting of normed spaces.

When the chosen norm is a Euclidean norm (i.e., formed from an inner product),
determining whether any given framework is globally rigid is an NP-hard problem \cite{saxe1980}. This complexity is reduced to polynomial-time if we restrict to \emph{generic} frameworks;
those where the $d|V|$ coordinates of $p$ are pair-wise distinct and algebraically independent over the field $\mathbb{Q}$.
With an understanding of the basic principals of rigidity theory,
it is easy to see that rigidity in this case is a generic property, in that a single generic framework $(G,p)$ in $d$-dimensional Euclidean space is rigid if and only all other generic frameworks $(G,q)$ in $d$-dimensional Euclidean space are rigid (see \cite{ar78}).
A significantly more difficult result was achieved by Gortler, Healy and Thurston \cite{ght10}, who proved that global rigidity in Euclidean spaces is also a generic property.

Much of the intuition that is built up when considering rigidity and global rigidity in Euclidean spaces vanishes when we turn to \emph{polyhedral normed spaces} (denoted by $(\mathbb{R}^d,\|\cdot\|_\mathcal{P})$) -- normed spaces whose unit ball is a polytope (denoted by $\mathcal{P}$).
For example, a framework $(G,p)$ in $(\mathbb{R}^d,\|\cdot\|_\mathcal{P})$ ($d \geq 2$) with $G=(\{a,b\},\{ab\})$ and $p(a)\neq p(b)$ is no longer rigid; a striking contrast to the Euclidean case where the same framework would be rigid.
To see this, if we pin the vertex $a$ at $p(a)$, the vertex $b$ can still traverse a sphere (with respect to the $\|\cdot\|_\mathcal{P}$ metric) centred around $p(a)$ in a faux-rotational movement; however, no such continuous family of isometries exists for $(\mathbb{R}^d,\|\cdot\|_\mathcal{P})$. See \Cref{fig:square} for an example of this framework when $\mathcal{P}$ is a square.
A major difficulty when working with polyhedral normed spaces is that rigidity is no longer a generic property.
This is as, unlike the Euclidean case, every graph will have a non-empty open set of flexible realisations (see \Cref{prop:genericflexible}).
For more on polyhedral normed space rigidity, see Kitson's foundational paper on the topic \cite{Kitson15}.

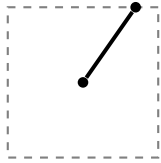
\begin{figure}[tp]
	\begin{center}
		\begin{tikzpicture}[scale=2]
			\draw[dashed,gray,thick] (0.5,0.5) -- (-0.5,0.5) -- (-0.5,-0.5) -- (0.5,-0.5) -- (0.5,0.5);
			
			\node[vertex] (1) at (0,0) {};
			\node[vertex] (2) at (0.35,0.5) {};
			
			\draw[edge] (1)edge(2);
		\end{tikzpicture}
	\end{center}
	\caption{A flexible framework in $\ell_\infty^2$ which is rigid under the standard Euclidean metric. This is solely due to the lack of isometry of $\ell_\infty^2$ which fixes the centre vertex and takes the upper vertex to any point on the dashed square.}\label{fig:square}
\end{figure}

The aim of this paper is to investigate the properties of global rigidity for polyhedral normed spaces, with a particular focus on $\ell_1^d= (\mathbb{R}^d,\|\cdot\|_1)$ and $\ell_\infty^d= (\mathbb{R}^d,\|\cdot\|_\infty)$,
the normed spaces whose unit balls correspond to the $d$-dimensional cross-polytope and the $d$-dimensional hypercube respectively.
These two particular norms have a wide variety of applications throughout mathematics;
specifically, the $\ell_1$ norm is vital to LASSO regression, an important tool in statistics \cite{tbi96}, and the space $\ell_\infty^d$ contains all metric spaces with $d$ points \cite{Frechet1910}.
Both $\ell_1$ and $\ell_\infty$ norms have also been investigated in the setting of distance geometry problems due to their computability \cite{crippen,leo17}.

\subsection{Main results}\label{subsec:main}

The main results of this paper can be split into three categories;
exact characterisations of global rigidity, necessary conditions for global rigidity, and sufficient conditions for global rigidity.
%The various definitions surrounding monochromatic subgraphs and the directed colouring $\phi_p$ can be found in \Cref{subsec:poly} and \Cref{subsec:linf}.)

\subsubsection*{Exact conditions for global rigidity in general polyhedral normed spaces and $\ell_1^d$}

To better understand global rigidity in a polyhedral normed space $(\mathbb{R}^d,\|\cdot\|_{\mathcal{P}})$ for a given graph $G=(V,E)$, we introduce \emph{directed colourings};
vectors $(\phi(e))_{e \in E}$ where each $\phi(e)$ is a normal vector for a face of the polytope $\mathcal{P}$.
Here we denote the set of directed colourings of $G$ by $\Phi_G$.
Any framework $(G,p)$ naturally generates a set $\Phi_p$ of directed colourings:
assuming a fixed ordering of the vertices of $G$,
we choose $\phi(vw)$ to be the normal vector for a face of $\mathcal{P}$ containing the vector $\frac{p(v)-p(w)}{\|p(v)-p(w)\|_{\mathcal{P}}}$ (or the zero vector if $p(v)=p(w)$).
If the set $\Phi_p$ contains a single element, we say that the framework $(G,p)$ is \emph{well-positioned}, and we set $\phi_p$ to be the unique element in $\Phi_p$.
(These concepts are explained in more detail in \Cref{sec:poly}.)
Unlike properties such as genericity, checking if a framework is well-positioned in polyhedral normed space can be performed in polynomial time.
Similarly, rigidity can be checked in polynomial time for well-positioned frameworks in polyhedral normed spaces, since a well-positioned framework is rigid if and only if the matrix $M(G,\phi_p)$ (an analogue of the rigidity matrix defined in \Cref{subsec:normrigid}) has rank $d|V|-d$ (\Cref{t:dk}, originally proved in \cite{Kitson15}).

Our first main result of the paper links directed colourings and global rigidity for polyhedral normed spaces.

\begin{theorem}\label{thm:compute}
	Let $(G,p)$ be a well-positioned and rigid framework in a polyhedral normed space $(\mathbb{R}^d,\|\cdot \|_\mathcal{P})$.
	Then the following two properties are equivalent:
	\begin{enumerate}
		\item \label{thm:compute1} $(G,p)$ is globally rigid in $(\mathbb{R}^d,\|\cdot \|_\mathcal{P})$;
		\item \label{thm:compute2} for each $\phi \in \Phi_G$ where $f_G(p)$ is in the column space of $M(G,\phi)$,
		either $\phi$ is \emph{isometric} to $\phi_p$ -- in that there exists a linear isometry $T$ of $(\mathbb{R}^d,\|\cdot \|_\mathcal{P})$ where $\phi = T \circ \phi_p$ -- or
		\begin{equation*}
			f_G(p) \notin f_G\left(M(G,\phi)^{-1}(f_G(p)) \right).
		\end{equation*}
	\end{enumerate}
\end{theorem}

It follows from \Cref{thm:compute} that there exists a deterministic algorithm for checking if a framework is globally rigid in any given polyhedral normed space.
It is worth noting that this algorithm has the potential to be slow. This is solely due to the size of the set $\Phi_G$, which contains $(\# \textrm{ faces of } \mathcal{P})^{|E|}$ elements that need to be checked.
A polynomial-time algorithm is actually possible when checking global rigidity in $\ell_\infty^2$ for frameworks with a complete graph (\Cref{thm:fastalgcomplete}),
however we prove that determining if a framework is globally rigid in $\ell_\infty^d$ is NP-Hard (\Cref{thm:nphard}).

%An alternative characterisation of \Cref{thm:compute} is achievable when we restrict to $\ell_1$ normed spaces (\Cref{t:farkas}).
%This is done by the use of Farkas' lemma, and was originally suggested to the author by Shin-ichi Tanigawa.

\subsubsection*{Necessary conditions for global rigidity in polyhedral normed spaces}

Hendrickson's condition for a normed space $(\mathbb{R}^d,\|\cdot\|)$ can be stated as follows: any generic globally rigid framework in $(\mathbb{R}^d,\|\cdot\|)$ is \emph{redundantly rigid}, in that the removal of any edge preserves rigidity.
(Here ``generic'' can be interpreted as whatever is suitable for the chosen normed space; see \Cref{subsec:generic} for the polyhedral norm variant.) 
The condition's namesake stems from the famous result of Hendrickson \cite{hendrickson1992}, where it was proven that Hendrickson's condition holds in Euclidean spaces.
It was later shown by the author, Hewetson and Nixon that Hendrickson's condition also holds for analytic normed spaces \cite[Theorem 3.8]{DHN22}.

As we discuss a little later, Hendrickson condition is not the correct necessary condition for polyhedral normed spaces.
Instead, we turn to a strengthened variant of Hendrickson's condition that was proven by Garamvölgyi, Gortler and Jordán \cite{ggj22}: if a graph is globally rigid in $d$-dimensional Euclidean space, then its edge set forms a connected set in the $d$-dimensional rigidity matroid.
We prove a similar result for polyhedral normed spaces,
with the $d$-dimensional rigidity matroid replaced with the \emph{$(d,d)$-sparsity matroid} (see \Cref{subsec:matroid} for more details).
Here, we say that a graph is \emph{$\mathcal{M}(d,d)$-connected} if its edge set is connected in the $(d,d)$-sparsity matroid.

\begin{theorem}\label{t:globness}
    Let $(G,p)$ be a generic globally rigid framework in $(\mathbb{R}^d,\|\cdot\|_{\mathcal{P}})$.
    Then $G$ is $\mathcal{M}(d,d)$-connected.
\end{theorem}

One important consequence of \Cref{t:globness} is the following: if a graph is minimally rigid in a polyhedral normed space,
then it is not globally rigid in the same polyhedral normed space.
While this seems intuitive, the usual tools used to prove this in Euclidean space fail when applied to polyhedral normed spaces.
Specifically, Hendrickson's proof relies on the generic fibres of each rigidity map $f_G$ being smooth manifolds for the Euclidean norm, which is no longer true for polyhedral norms.

\subsubsection*{Sufficient conditions for global rigidity in \texorpdfstring{$\ell^d_\infty$}{L-infinity}}

Our first sufficient condition links global rigidity in smooth $\ell_p$ spaces (i.e., where $p \in (1,\infty)$) and $\ell_\infty$.

\begin{theorem}\label{t:lpapprox}
	Let $(G,p)$ be a generic framework in $\mathbb{R}^d$. Suppose that:
	\begin{enumerate}
		\item $(G,p)$ is rigid in $\ell_\infty^d$;
		\item there exists an increasing  sequence $(s_n)_{n \in \mathbb{N}}$ of even integers such that $(G,p)$ is globally rigid in $\ell_{s_n}^d$ for each $n \in \mathbb{N}$.
	\end{enumerate}
	Then $(G,p)$ is globally rigid in $\ell_\infty^d$.
\end{theorem}

\Cref{t:lpapprox} is highly applicable to the study of global rigidity in $\ell_\infty^d$.
One important application is the following.

\begin{theorem}\label{t:linfplane}
	Let $(G,p)$ be a generic rigid framework in $\ell_\infty^2$.
	Then $(G,p)$ is globally rigid in $\ell_\infty^2$ if and only if $G$ is $\mathcal{M}(2,2)$-connected.
\end{theorem}

It was proven by Kitson and Power that every graph containing two edge-disjoint spanning trees has a generic rigid realisation in $\ell_\infty^2$ \cite[Theorem 4.10]{KP14}.
Because all $\mathcal{M}(2,2)$-connected graphs contain two edge-disjoint spanning trees, we immediately obtain the following corollary.

\begin{corollary}\label{cor:linfplane}
	 A graph has a generic globally rigid realisation in $\ell_\infty^2$ if and only if it is $\mathcal{M}(2,2)$-connected.
\end{corollary}

(We remark here that, since $\ell_1^2$ and $\ell_\infty^2$ are isometrically isomorphic,
both \Cref{t:linfplane} and \Cref{cor:linfplane} are true if we replace $\ell_\infty^2$ with $\ell_1^2$.)

\Cref{cor:linfplane} has an additional surprising application.
It is a simple exercise to show that any graph that is redundantly rigid in $\ell_\infty^2$ can be decomposed into two edge-disjoint spanning subgraphs that are 2-edge-connected (see \Cref{t:dk}),
which in turn implies that the graph has at least twice as many edges as vertices.
However, there exist $\mathcal{M}(2,2)$-connected graphs that have  strictly less edges: for example, if $G$ is a 2-connected $(2,2)$-circuit (defined in \Cref{subsec:matroid}), then it is $\mathcal{M}(2,2)$-connected (and hence has a globally rigid realisation in $\ell_\infty^2$) but $|E| =2|V|-1 < 2|V|$ (and hence has no redundantly rigid realisation in $\ell_\infty^2$).
This observation can be rather pithily summarised as follows:

\begin{theorem}\label{t:hend}
	Hendrickson's condition is false in $\ell_\infty^2$ in a strong sense, in that there exist graphs with generic globally rigid realisations in $\ell_\infty^2$ but no redundantly rigid realisations in $\ell_\infty^2$. 
\end{theorem}

Our last main sufficient condition for global rigidity involves a slightly  different approach. Here we instead focus on a special type of directed colouring that we call \emph{strong} (see \Cref{def:strong}).
Using this, we are able to prove the following sufficient condition for global rigidity in $\ell_\infty^d$.

\begin{theorem}\label{t:main2}
    Let $(G,p)$ be a generic framework in $\mathbb{R}^d$.
    Suppose that every monochromatic subgraph of $(G,\phi_p)$ is 2-connected.
    Then $(G,p)$ is globally rigid in $\ell_{\infty}^d$.
\end{theorem}

Every every monochromatic subgraph of a well-positioned framework is 2-connected if and only if the framework is \emph{vertex-redundantly rigid in $\ell_\infty^d$} (i.e., deleting any vertex from the framework always produces a rigid framework).
Because of this,
\Cref{t:main2} is an analogue of \cite[Theorem 4.2]{t15} for global rigidity in $\ell_\infty^d$.

\subsection{Direct application: network location in grid-system urban areas}

A commonly cited application for Euclidean rigidity theory is network localisation:
given a set of nodes where the position of some nodes is known and some pairwise distances between nodes are known,
determine the positions of all the nodes.
A motivating example of this would be a network of radio transceivers where some transceivers have fixed and known locations (such as radio towers), while others are portable with unknown locations. Here pairwise distances can be determined between sufficiently close radio transceivers by measuring the time it takes for a signal to travel from one transceiver to the another and back again.
This motivating example runs into issues for urban environments due to Non-line-of-sight (NLOS) radio signal propagation. In layman's terms, the radio signals are ``funnelled'' down the various streets of the urban area, and so the distance being measured is distorted.

To address the NLOS radio signal propagation issue,
Chiu and Chen \cite{chiuchen} studied network localisation in urban areas with grid street plans -- such as New York City or (for an example closer to the author's home) Milton Keynes -- using a mixture of $\ell_1$ and $\ell_2$ metric measurements.
By combining their model with \Cref{t:linfplane},
we can reconstruct radio transceiver networks in urban environments if the following criteria are met:
\begin{enumerate}
	\item all nodes are assumed to have generic positions;
	\item the environment is an urban environment with a suitably fine and consistent grid system;
	\item at least one node's location is known;
	\item the framework generated from the network is rigid in $\ell_1^2$;
	\item the graph generated from the network is $\mathcal{M}(2,2)$-connected;
	\item if exactly one node's location is known, we additionally need to know the general direction between two pairs of nodes (for example, roughly knowing if one node is north/east/south/west of another node) so as to remove reflections and 90 degree rotations.
\end{enumerate}

\subsection{Layout of paper}

In \Cref{sec:prelim},
we describe some important background results in the study of normed space rigidity,
as well as describing the $(d,d)$-sparsity matroid.
In \Cref{sec:poly},
we introduce polyhedral normed spaces and describe the various tools we shall be using throughout.
We also prove some important results regarding framework rigidity and redundant rigidity, particularly in $\ell_\infty$ normed spaces.
In \Cref{sec:equiv},
we first prove \Cref{thm:compute},
then we prove that determining if an arbitrary framework in $\ell_\infty^d$ is globally rigid is NP-Hard.
%an equivalence condition for global rigidity in $\ell_1^d$ that uses Farkas' lemma.
In \Cref{sec:ness},
we prove that, if a framework $(G,p)$ is globally rigid and well-positioned in a polyhedral normed space $(\mathbb{R}^d,\|\cdot\|_\mathcal{P})$,
then $(G,p)$ is also globally rigid and well-positioned in some analytic normed space $(\mathbb{R}^d,\|\cdot \|)$ which shares similar properties to $\| \cdot\|_{\mathcal{P}}$.
From this, we use previously known results regarding global rigidity in analytic normed spaces to prove \Cref{t:globness}.
In \Cref{sec:linfplane},
we prove \Cref{t:lpapprox}. We use this result to first prove \Cref{t:linfplane}, then we use it to show that that projecting any generic rigid framework in $\ell_\infty^{d+1}$ into $\ell_\infty^d$ creates a generic globally rigid framework in $\ell_\infty^d$ (\Cref{t:project}).
This latter result is in turn applied to prove that all complete graphs with at least $2d+2$ vertices are globally rigid in $\ell_\infty^d$ (\Cref{cor:k2d2}).
In \Cref{sec:suff},
we describe a different sufficient condition for global rigidity involving strong directed colourings.
Using this, we then prove \Cref{t:main2}.

\subsection{Notation}

We use the shorthand notation $[n] := \{1,\ldots,n\}$.
If the vertex/edge set of a graph $G$ is not explicitly stated, they are assumed to be $V$ and $E$ respectively.
For a given realisation $p:V \rightarrow \mathbb{R}^d$,
we denote the vector value at the vertex $v \in V$ by $p(v) = (p_1(v),\ldots,p_d(v)$.
We apply an affine transformation $g : \mathbb{R}^d \rightarrow \mathbb{R}^d$ to the realisation $p$, here denoted $g \circ p$, by setting $(g \circ p)(v) = g(p(v))$ for each $v \in V$.
Throughout the paper, we fix the dot product $x\cdot y = \sum_{i=1}^d x_i y_i$ for any vectors $x=(x_1,\ldots,x_d)$ and $y=(y_1,\ldots,y_d)$ in $\mathbb{R}^d$.
Furthermore, all normed spaces throughout are both finite dimensional and real, i.e., of the form $(\mathbb{R}^d,\|\cdot\|)$.
We reserve $\mathbf{0}$ to represent any vector with all coordinates being zero.
We set $\leq$ be the coordinate-wise partial ordering on any linear space $\mathbb{R}^n$;
i.e.~$x \leq y$ if and only if $x_i \leq y_i$ for each $i \in [n]$.

\section{Preliminaries on normed space rigidity}
\label{sec:prelim}

We begin by outlining some of the basic results in normed space rigidity theory we require throughout the paper.

\subsection{Supports, smooth points and isometries for normed spaces}
Given a normed space $(\mathbb{R}^d,\|\cdot\|)$ and a vector $x \in \mathbb{R}^d \setminus \{0\}$,
we define a vector $y$ to be a \emph{support vector} of $x$ (with respect to the norm $\| \cdot\|$) if $x \cdot y = \|x\|$ and $x' \cdot y \leq \|x\|$ for all $x' \in \mathbb{R}^d$ with $\|x'\|\leq \|x\|$.
It follows from the Hahn-Banach theorem that every non-zero vector has a support vector.
If a non-zero point $x$ has a unique support vector (which we denote by $\varphi(x)$) then we say it is \emph{smooth} (with respect to $\|\cdot\|$);
in fact, the support vector of a smooth point is exactly the gradient of the norm $\|\cdot\|$ at the point $x$ (see, for example, \cite[Lemma 1]{ks15}).

As every norm is a convex function,
the set of smooth points of a normed space $(\mathbb{R}^d,\|\cdot\|)$ --
here denoted $\smooth (\mathbb{R}^d,\|\cdot\|)$ -- has a Lesbegue measure zero complement and the \emph{dual map}
\begin{equation*}
	\varphi : \smooth (\mathbb{R}^d,\|\cdot\|) \rightarrow \mathbb{R}^d, ~ x \mapsto \varphi(x),
\end{equation*}
is the continuous derivative of the norm $\| \cdot\|$ over the set $\smooth (\mathbb{R}^d,\|\cdot\|)$ \cite[Theorem 25.5]{rockafellar}.
If every non-zero point of a normed space is smooth (i.e., $\smooth (\mathbb{R}^d,\|\cdot\|) = \mathbb{R}^d \setminus \{0\}$), then we say that the normed space itself is \emph{smooth}.
Examples of smooth spaces include any $\ell_p$ normed space with $p \in (1,\infty)$.

It was shown by Mazur and Ulam \cite{mazurulam} that every isometry from a normed space to itself is a bijective affine transformation.
Thus the set of isometries for a normed space $(\mathbb{R}^d,\|\cdot\|)$ --
here denoted $\Isom (\mathbb{R}^d,\|\cdot\|)$ -- is a closed subgroup of the group of bijective affine transformations of $\mathbb{R}^d$ containing the family of translations of $\mathbb{R}^d$.
Since this latter group is a Lie group,
it follows from the Closed Subgroup Theorem \cite{cartan} that the group $\Isom (\mathbb{R}^d,\|\cdot\|)$ is a Lie group.
This in itself is not something we need to concern ourselves about too much (especially if you are somewhat allergic to Lie group terminology);
the main takeaway from this is that the tangent space of $\Isom (\mathbb{R}^d,\|\cdot\|)$ at the identity map $I_d$ -- here denoted $\mathcal{T}(\mathbb{R}^d,\|\cdot\|)$ -- is a well-defined linear subspace of bijective affine transformations,
and $d \leq \dim \mathcal{T}(\mathbb{R}^d,\|\cdot\|) \leq \binom{d+1}{2}$,
with $\dim \mathcal{T}(\mathbb{R}^d,\|\cdot\|)  = d$ if and only if $(\mathbb{R}^d,\|\cdot\|)$ has finitely many linear isometries.

We say that our normed space $(\mathbb{R}^d,\|\cdot\|)$ is \emph{Euclidean} if the norm $\|\cdot\|$ is an inner product norm.
This is equivalent to the condition that $\dim \mathcal{T}(\mathbb{R}^d,\|\cdot\|) = \binom{d+1}{2}$ \cite[Lemma 4]{ms43}.
Since all Euclidean spaces of the same dimension are isometrically isomorphic (i.e., there exists a linear isometry from one space to another),
we say that the norm
\begin{equation*}
	\|x\| = \sqrt{ \sum_{i=1}^d x_i^2} = \sqrt{x \cdot x}
\end{equation*}
is the \emph{standard Euclidean norm}.

\subsection{Rigidity for normed spaces}\label{subsec:normrigid}

Given a graph $G=(V,E)$ and a normed space $(\mathbb{R}^d,\|\cdot\|)$, we say that two frameworks $(G,p)$ and $(G,q)$ in $(\mathbb{R}^d,\|\cdot\|)$ are \emph{equivalent} if they have the same length edges with respect to the norm $\|\cdot\|$, which is the same as the equality $f_G(p) = f_G(q)$ holding.
If there exists an isometry $g \in \Isom (\mathbb{R}^d,\|\cdot\|)$ where $g \circ p = q$,
then we say that $(G,p)$ and $(G,q)$ are \emph{congruent}.
Note that any pair of congruent frameworks are also equivalent.

A framework $(G,p)$ is said to be \emph{well-positioned} if every vector $p(v) - p(w)$, $vw \in E$, is a smooth point in $(\mathbb{R}^d,\|\cdot\|)$;
equivalently, $(G,p)$ is well-positioned if and only $p$ is a smooth point of the rigidity map $f_G$.
Given a well-positioned framework $(G,p)$,
the \emph{rigidity matrix} of $(G,p)$ is the $|E| \times d|V|$ Jacobian matrix $df_G(p)$ of $f_G$ at the point $p$.
Specifically, the matrix $df_G(p)$ has an entry $df_G(p)_{e,(u,i)}$ for each $e=vw \in E$ and each pair $(u,i) \in V \times [d]$,
where
\begin{align*}
    df_G(p)_{e,(u,i)} = 
    \begin{cases}
        \varphi(p(v) - p(w))_i &\text{if } u =v\\
        \varphi(p(w) - p(v))_i &\text{if } u = w\\
        0 &\text{otherwise.}
    \end{cases}
\end{align*}
An \emph{infinitesimal flex} of $(G,p)$ is any element of $\ker df_G(p)$.
For every $g \in \mathcal{T}(\mathbb{R}^d,\|\cdot\|)$,
the vector $g \circ p$ is an infinitesimal flex (see, for example, \cite[Lemma 2.1]{KP14});
any such infinitesimal flex is then called a \emph{trivial infinitesimal flex} of $(G,p)$.

For a framework $(G,p)$ in $(\mathbb{R}^d,\|\cdot\|)$,
we define the following properties:
\begin{enumerate}
	\item $(G,p)$ is \emph{rigid} (also known as locally rigid) if there exists $\varepsilon >0$ such that any equivalent framework $(G,q)$ where $\|p(v)-q(v)\|<\varepsilon$ for each $v \in V$ is congruent to $(G,p)$.
	\item $(G,p)$ is \emph{globally rigid} if every equivalent framework $(G,q)$ is congruent to $(G,p)$.
	\item If $(G,p)$ is well-positioned\footnote{It is possible to define infinitesimal rigidity without requiring a framework is well-positioned; see, for example, \cite[Section 1.1]{Dewar22} or \cite[Section 2]{KP14}. Since the conditions is more complex and unnecessary for our goals, we have opted for the more streamlined definition.}, we say $(G,p)$ is \emph{infinitesimally rigid} if every infinitesimal flex is trivial.
	\item If $(G,p)$ is well-positioned, we say $(G,p)$ is \emph{redundantly rigid} if $(G-e,p)$ is infinitesimally rigid for every edge $e \in E$.
\end{enumerate}
It is clear that global rigidity implies rigidity.
What is less clear (but indeed still true) is that infinitesimal rigidity implies rigidity.

\begin{theorem}[Dewar {\cite[Theorem 3.7]{Dewar22}}]
	Every infinitesimally rigid framework is rigid.
\end{theorem}

Given integers $d,k$ with $k \leq \binom{d+1}{2}$, we say that a graph $G=(V,E)$ is \emph{$(d,k)$-sparse} if for every subset $X \subset V$ with $|X| \geq d$,
the induced subgraph $G[X]$ has at most $d|X| - k$ edges.
If $G$ is both $(d,k)$-sparse and $|E|=d|V|-k$, then $G$ is said to be \emph{$(d,k)$-tight}.
For example, a graph is $(1,1)$-sparse if and only if it is a forest, and $(1,1)$-tight if and only if it is a tree.
As can be seen by the following result, $(d,k)$-sparsity is closely tied to rigidity.

\begin{corollary}[{\cite[Corollary 4.13]{Dewar21}}]\label{t:maxwell}
	Let $(\mathbb{R}^d,\|\cdot\|)$ be a normed space and let $k$ be the dimension of its isometry group.
	If a framework $(G,p)$ in $(\mathbb{R}^d,\|\cdot\|)$ is infinitesimally rigid, then $G$ contains a $(d,k)$-tight spanning subgraph.
\end{corollary}

As all 1-dimensional normed spaces are Euclidean, it is easy to see that the necessary condition given in \Cref{t:maxwell} is actually sufficient when $d=1$.
In fact, this is also true when $d=2$.

\begin{theorem}[{Pollaczek-Geiringer \cite{PollaczekGeiringer}}]\label{t:laman}
	Let $(\mathbb{R}^2,\|\cdot\|)$ be Euclidean.
	Then the following properties are equivalent for any graph $G$:
	\begin{enumerate}
		\item there exists an infinitesimally rigid framework $(G,p)$ in $(\mathbb{R}^2,\|\cdot\|)$;
		\item $G$ contains a $(2,3)$-tight spanning subgraph, or $G$ is a single vertex.
	\end{enumerate}
\end{theorem}

\begin{theorem}[{Dewar \cite{Dewar20}}]\label{t:dewar}
	Let $(\mathbb{R}^2,\|\cdot\|)$ be non-Euclidean.
	Then the following properties are equivalent for any graph $G$:
	\begin{enumerate}
		\item there exists an infinitesimally rigid framework $(G,p)$ in $(\mathbb{R}^2,\|\cdot\|)$;
		\item $G$ contains a $(2,2)$-tight spanning subgraph.
	\end{enumerate}
\end{theorem}

\begin{remark}
	\Cref{t:dewar} was first proven by Kitson and Power for each space $\ell_q^2$ with $q \in [1,\infty] \setminus \{2\}$ \cite{KP14} and then by Kitson for polyhedral normed spaces by Kitson \cite{Kitson15}.
\end{remark}

It is known that the converse of \Cref{t:maxwell} is false when $d \geq 3$ and $\|\cdot\|$ is either a Euclidean norm (due to the infamous double-banana graph seen in \Cref{fig:db}), or a cylindrical norm \cite[Section 6]{kl20};
in the latter case, an alternative exact combinatorial is known however \cite{dk23}.
In most other families of norms, it is an open problem to determine an exact combinatorial condition for rigidity.

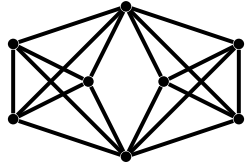
\begin{figure}[htp]
	\begin{center}
        \begin{tikzpicture}[scale=0.5]
			\node[vertex] (top) at (0,2) {};
			\node[vertex] (bottom) at (0,-2) {};

			\node[vertex] (11) at (-1,0) {};
			\node[vertex] (21) at (-3,1) {};
			\node[vertex] (31) at (-3,-1) {};
			
			\node[vertex] (12) at (1,0) {};
			\node[vertex] (22) at (3,1) {};
			\node[vertex] (32) at (3,-1) {};
			
			\draw[edge] (11)edge(top);
			\draw[edge] (21)edge(top);
			\draw[edge] (31)edge(top);
			
			\draw[edge] (top)edge(12);
			\draw[edge] (top)edge(22);
			\draw[edge] (top)edge(32);
   
			\draw[edge] (11)edge(bottom);
			\draw[edge] (21)edge(bottom);
			\draw[edge] (31)edge(bottom);
			
			\draw[edge] (bottom)edge(12);
			\draw[edge] (bottom)edge(22);
			\draw[edge] (bottom)edge(32);
			
			\draw[edge] (11)edge(21);
			\draw[edge] (21)edge(31);
			\draw[edge] (31)edge(11);
			
			\draw[edge] (12)edge(22);
			\draw[edge] (22)edge(32);
			\draw[edge] (32)edge(12);
		\end{tikzpicture}
	\end{center}
	\caption{The double-banana graph. The graph is $(3,6)$-tight, but clearly has no infinitesimally rigid realisations in 3-dimensional Euclidean space due to the vertex cut pair acting like a hinge.}\label{fig:db}
\end{figure}

\subsection{Matroidal connectivity and redundant rigidity}
\label{subsec:matroid}

For any integers $d,k$ where $d \geq 1$ and $0 \leq k \leq 2d-1$, and for any set $V$,
take $\binom{V}{2} := \{vw :v,w \in V, v \neq w\}$ to be the base set and define the set of subsets
\begin{equation*}
	\mathcal{E} = \left\{ E \subset \binom{V}{2} : (V,E) \text{ is a $(d,k)$-sparse graph} \right\}.
\end{equation*}
It was proven by Lee and Streinu \cite{LS08} that the pair $(\binom{V}{2},\mathcal{E})$ always forms a matroid.
We call this matroid the \emph{$(d,k)$-sparsity matroid} (with vertex set $V$), and we denote it by $\mathcal{M}(d,k)$.
So long as $\binom{|V|}{2} \geq d|V| - k$,
the bases of $\mathcal{M}(d,k)$ are exactly the edge sets that induce $(d,k)$-tight subgraphs.
It is possible to determine the rank of an edge set in the $(d,k)$-sparsity matroid with vertex set $V$ in $O(|V|^2)$ time using the pebble game algorithm defined in \cite{LS08}.

For a graph $G=(V,E)$ and integers $d,k$ where $d \geq 1$ and $0 \leq k \leq 2d-1$,
we define the following properties:
\begin{enumerate}
	\item $G$ is a \emph{$(d,k)$-circuit} if $E$ is a circuit of $\mathcal{M}(d,k)$.
	\item $G$ is \emph{$(d,k)$-redundant} if every edge is contained in a $(d,k)$-circuit.
	\item $G$ is \emph{$\mathcal{M}(d,k)$-connected} if every pair of edges are contained within a $(d,k)$-circuit.
\end{enumerate}
It is easy to see that any $\mathcal{M}(d,k)$-connected graph is $(d,k)$-redundant,
however the converse is not always true; indeed, $(d,k)$-redundant graphs do not even need to be connected.
Using ideas of Kir\'{a}ly and Mih\'{a}lyk\'{o} \cite{KM22},
we can link the two concepts when $d=k$ by introducing 2-connectivity as a required condition. 

\begin{lemma}\label{l:km}
	For any positive integer $d$,
	a graph is $\mathcal{M}(d,d)$-connected if and only if it is $(d,d)$-redundant and 2-connected.
\end{lemma}

\begin{proof}
	It was shown in \cite[Lemma 3.2]{KM22} that any graph that is $(d,d)$-redundant and 2-connected is $\mathcal{M}(d,d)$-connected.
	Every $(d,d)$-circuit is 2-connected (see, for example, \cite[Proposition 2.3]{KM22}),
	and so every $\mathcal{M}(d,d)$-connected graph is 2-connected.
	Finally, every $\mathcal{M}(d,d)$-connected graph is $(d,d)$-redundant by definition.
\end{proof}

We now wish to link $(d,d)$-redundancy and redundant rigidity.
It is important to note here that \textbf{a graph $G$ having redundantly rigid realisation is a stronger condition than $G-e$ having an infinitesimally rigid realisation for every edge $e$}.
While these two properties are the same in Euclidean spaces and other ``well-behaved'' normed spaces,
they are not the same for polyhedral normed spaces (see \Cref{ex:k5} later).

\begin{proposition}\label{p:connequiv}
	If a well-positioned framework $(G,p)$ is redundantly rigid in a normed space $(\mathbb{R}^d,\|\cdot\|)$ with a finite number of linear isometries, 
	then $G$ is $(d,d)$-redundant.
	If $G$ is also 2-connected, then $G$ is $\mathcal{M}(d,d)$-connected.
\end{proposition}

\begin{proof}
	For every edge $e \in E$, the framework $(G-e,p)$ is infinitesimally rigid.
	By \Cref{t:maxwell},
	each graph $G-e$ contains a spanning $(d,d)$-tight subgraph $H_e$.
	Note that $H_e +e$ is not $(d,d)$-sparse in $\mathcal{M}(d,d)$ while $H_e$ is $(d,d)$-sparse,
	and so there exists a $(d,d)$-circuit contained in $G$ that contains the edge $e$.
	Thus $G$ is $(d,d)$-redundant.
	The second part of the result now follows from \Cref{l:km}.
\end{proof}

The converse of \Cref{p:connequiv} is not true however;
the complete graph $K_5$ is $\mathcal{M}(2,2)$-connected, but (as we show later in \Cref{ex:k5}) not redundantly rigid in $\ell_\infty^2$.

\section{Rigidity for polyhedral normed spaces}\label{sec:poly}

In this section, we now focus on a specific family of normed spaces: polyhedral normed spaces.
Many of the results covered in this section were originally proven by Kitson in \cite{Kitson15}.
Our terminology is similar to Kitson's, but we do introduce new concepts such as directed colourings.

\subsection{Polyhedral normed spaces}\label{subsec:poly}

Let $\mathcal{P}$ be a compact convex polytope in $\mathbb{R}^d$ with non-empty interior.
This is equivalent to the following construction:
take a finite linearly spanning set $\mathcal{F} \subset \mathbb{R}^d \setminus \{\mathbf{0}\}$ and define
\begin{align*}
	\mathcal{P} := \left\{ x \in \mathbb{R}^d : ~ f \cdot x \leq 1 \text{ for all } f \in \mathcal{F} \right\},
\end{align*}
where $\cdot$ is the dot product of $\mathbb{R}^d$.
If the set $\mathcal{F}$ is minimal (in the sense that no subset can generate the same polytope),
then $\mathcal{F}$ is the set of normal vectors for the faces of $\mathcal{P}$ of dimension $d-1$.
Because of this we shall always assume $\mathcal{F}$ (now referred to as the \emph{faces} of $\mathcal{P}$) is minimal.
We are particularly interested in the case where $\mathcal{P}$ is \emph{centrally symmetric},
i.e., $-\mathcal{P} = \mathcal{P}$.
A convex polytope is centrally symmetric if and only if $-\mathcal{F} = \mathcal{F}$.

Given a centrally symmetric polytope $\mathcal{P}$,
we define the norm $\|\cdot \|_\mathcal{P}$ of $\mathbb{R}^d$ by setting
\begin{align*}
	\|\cdot \|_\mathcal{P} := \max_{f \in \mathcal{F}} f \cdot x
\end{align*}
for each $x \in \mathbb{R}^d$.
Any such pair $(\mathbb{R}^d,\|\cdot \|_{\mathcal{P}})$ is called a \emph{polyhedral normed space}.
Since any linear isometry of a normed space must send the extremal points of the unit ball to the extremal points,
it follows that every polyhedral normed space has finitely many linear isometries.

\begin{example}
	Two classic examples of polyhedral norms are $\ell_1^d := (\mathbb{R}^d,\|\cdot\|_1)$ and $\ell_\infty^d := (\mathbb{R}^d,\|\cdot\|_\infty)$,
	where
	\begin{equation*}
		\|(x_1,\ldots,x_d)\|_1 := \sum_{i=1}^d |x_i| , \qquad \|(x_1,\ldots,x_d)\|_\infty := \max_{1\leq i \leq d} |x_i|.
	\end{equation*}
	The unit balls for $\ell_1^d$ and $\ell_\infty^d$ are the $d$-dimensional cross-polytope and the $d$-dimensional hypercube respectively.
\end{example}

Now fix a polyhedral normed space $(\mathbb{R}^d,\|\cdot\|_\mathcal{P})$ with faces $\mathcal{F}$.
Let $G=(V,E)$ be a graph and define 
\begin{equation*}
	\Phi_G = \{ \phi = (\phi(e))_{e \in E} : \phi(e) \in \mathcal{F} \cup \{\mathbf{0}\} \text{ for all } e \in E \}.
\end{equation*}
We shall refer to $\phi \in \Phi_G$ as a \emph{directed colouring} of $G$ (with respect to $(\mathbb{R}^d,\|\cdot\|_\mathcal{P})$).

Fix an ordering on the vertex set of $G=(V,E)$.
For each $\phi \in \Phi_G$,
define the $|E| \times d|V|$ matrix $M(G,\phi)$ where for every $e=vw \in E$ with $v < w$, every $u \in V$ and every $i \in \{1,\ldots,d\}$ we have (using the coordinate notation $\phi(e) = (\phi(e)_1,\ldots,\phi(e)_d)$)
\begin{equation*}
    M(G,\phi)_{e,(u,i)} = 
    \begin{cases}
        \phi(e)_i &\text{if } u =v\\
        -\phi(e)_i &\text{if } u = w\\
        0 &\text{otherwise.} \\
    \end{cases}
\end{equation*}
Kitson \cite{Kitson15} proved that these matrices can be used to directly define the rigidity map.
Before we state his result using directed colourings,
we make the following observation.
For a framework $(G,p)$ in $(\mathbb{R}^d,\|\cdot\|_\mathcal{P})$,
fix $S := \{ M(G,\phi)p : \phi \in \Phi_G\}$.
We now note that if we choose $x,y \in S$ such that $x = M(G,\phi)p$ and $y = M(G,\phi')p$,
then there exists $z \in S$ such that $z \geq x$ and $z \geq y$.
Here we have $z = M(G,\phi'')p$ where $\phi'' \in \Phi_G$ is the directed colouring where for each edge $e=vw$, $v <w$, we set $\phi''(e) = \phi(e)$ if $\phi(e) (p(v)-p(w)) \geq \phi'(e) (p(v)-p(w))$,
and $\phi''(e) = \phi'(e)$ otherwise.
Hence the set $S$ has a unique maximal element under the coordinate-wise partial ordering.

\begin{proposition}[Kitson \cite{Kitson15}]
	Let $(G,p)$ be a framework in a polyhedral normed space $(\mathbb{R}^d,\|\cdot\|_\mathcal{P})$.
	Then
	\begin{equation*}
		f_G(p) = \max_{\phi \in \Phi_G} M(G,\phi)p.
	\end{equation*}
	Furthermore,
	$(G,p)$ is well-positioned if and only if the set
	\begin{equation*}
		\Phi_p := \left\{ \phi \in \Phi_G : M(G,\phi)p = f_G(p) \right\}
	\end{equation*}
	contains exactly one directed colouring (now denoted $\phi_p$),
	in which case the Jacobian of $f_G$ at $p$ is exactly the matrix $M(G,\phi_p)$.
\end{proposition}

We can restate another result of Kitson regarding how infinitesimal rigidity and rigidity coincide in polyhedral normed spaces using directed colourings.

\begin{theorem}[Kitson \cite{Kitson15}]\label{t:dk}
	Let $(G,p)$ be a well-positioned framework in $(\mathbb{R}^d,\|\cdot\|_\mathcal{P})$ with $G=(V,E)$.
	Then the following conditions are equivalent:
	\begin{enumerate}
		\item $(G,p)$ is rigid.
		\item $(G,p)$ is infinitesimally rigid.
		\item $\rank M(G,\phi_p) = d|V|-d$.
		\item $\ker M(G,\phi_p) = \{ (x)_{v \in V} : x \in \mathbb{R}^d\}$.
	\end{enumerate}
\end{theorem}

Sufficient conditions for rigidity are still possible for not well-positioned frameworks.

\begin{lemma}\label{l:local}
	Let $(G,p)$ be a framework in $(\mathbb{R}^d,\| \cdot\|_{\mathcal{P}})$.
	Then there exists an open neighbourhood $U$ of $p$ such that
	\begin{equation*}
		f_G^{-1}(f_G(p)) \cap U \subset \left( \bigcup_{\phi \in \Phi_p} (p + \ker M(G,\phi)) \right) \cap U.
	\end{equation*}
\end{lemma}

\begin{proof}
	We first note that there exists some $\varepsilon >0$ such that $f_G(p) - M(G,\phi)p >\varepsilon$ for all $\phi \in \Phi_G \setminus \Phi_p$.
	Since $f_G$ is continuous, we can choose a small neighbourhood $U$ of $p$ such that $f_G(q) - M(G,\phi)q >\varepsilon/2$ for all $\phi \in \Phi_G \setminus \Phi_p$ and $q \in U$.
	Hence for all $q \in f_G^{-1}(f_G(p)) \cap U$ we have that $\Phi_q \subset \Phi_p$,
	and so $q \in p+ \ker M(G,\phi)$ for some $\phi \in \Phi_p$.
\end{proof}

\begin{proposition}\label{p:phirigid}
	Let $(G,p)$ be a framework in $(\mathbb{R}^d,\| \cdot\|_{\mathcal{P}})$.
	If $\rank M(G,\phi) = d|V| -d$ for all $\phi \in \Phi_p$,
	then $(G,p)$ is rigid.
\end{proposition}

\begin{proof}
	Given $X := \{ (x)_{v \in V} : x \in \mathbb{R}^d\}$,
	we have that $\ker M(G,\phi) = X$ for each $\phi \in \Phi_p$.
	By \Cref{l:local},
	there exists an open neighbourhood $U$ of $p$ where $f_G^{-1}(f_G(p)) \cap U \subset (p + X) \cap U$,
	hence $(G,p)$ is rigid.
\end{proof}

It should be noted that the converse of \Cref{p:phirigid} is not true.
See \cite[Example 6.3]{Dewar22} for a specific example of a framework in $\ell_\infty^2$ that is rigid but does not satisfy the rank condition stipulated in \Cref{p:phirigid}.

\subsection{Generic frameworks}\label{subsec:generic}

A framework $(G,p)$ in $(\mathbb{R}^d,\|\cdot\|_\mathcal{P})$ is \emph{generic} if the coordinates of $p$ are distinct and algebraically independent over the field $\mathbb{Q}(\mathcal{P})$ (i.e.~the smallest field that contains $\mathbb{Q}$ and the coordinates of the vectors in $\mathcal{F}$).
Every generic framework is also well-positioned:
if there exist distinct $\phi,\phi' \in \Phi_p$ where $M(G,\phi)p=M(G,\phi')p$,
then $p$ is a zero of the non-zero matrix $(M(G,\phi)-M(G,\phi'))$ with entries in $\mathbb{Q}( \mathcal{P})$.

\begin{remark}
	By abuse of terminology, we say that a framework $(G,p)$ in $\mathbb{R}^d$ (with no discussion over the norm being implemented) is \emph{generic} if the coordinates of $p$ are distinct and algebraically independent over $\mathbb{Q}$.
	If the faces of a polytope $\mathcal{P}$ are rational (i.e., $\mathbb{Q}(\mathcal{P})=\mathbb{Q}$), then $(G,p)$ is generic as a framework in $(\mathbb{R}^d, \| \cdot\|_\mathcal{P})$ if and only if it is generic as a framework in $\mathbb{R}^d$.
\end{remark}

For an $m \times n$ matrix $M$,
we define $\col M \subset \mathbb{R}^m$ to be the column space of $M$.
The column space of any matrix is equal to both its image and the orthogonal complement of the matrix's cokernel (i.e.~the left null space $\ker M^T$).

\begin{lemma}\label{l:generic}
    Let $(G,p)$ be a generic framework in $(\mathbb{R}^d,\|\cdot\|_\mathcal{P})$.
    If $f_G(p) \in \col M(G,\phi)$ then $\col M(G,\phi_p) \subseteq \col M(G,\phi)$.
    Furthermore,
    if $(G,p)$ is rigid then $\col M(G,\phi_p) = \col M(G,\phi)$.
\end{lemma}

\begin{proof}
    Since every coordinate of $M(G,\phi)$ lies in $\mathbb{Q}(\mathcal{P})$,
    there exists a basis $b_1,\ldots, b_n$ of the linear space $\ker M(G,\phi)^T$ such that the coordinates of each $b_i$ lie in the field $\mathbb{Q}(\mathcal{P})$.
    Since $f_G(p) \in \col M(G,\phi)$, it follows that $b_i^T M(G,\phi_p) p = b_i^T f_G(p) = 0$ for each $i \in \{1,\ldots,n\}$.
    By our choice of $b_1,\ldots, b_n$,
    each matrix $b_i^T M(G,\phi_p)$ has coefficients in $\mathbb{Q}(\mathcal{P})$.
    As the coordinates of $p$ are algebraically independent over $\mathbb{Q}(\mathcal{P})$,
    it follows that $ M(G,\phi_p)^T b_i = \mathbf{0}$ for each $i \in \{1,\ldots,n\}$.
    Hence $\ker M(G,\phi)^T \subset \ker M(G,\phi_p)^T$,
    and so $\col M(G,\phi_p) \subseteq \col M(G,\phi)$.
    
    Suppose that $(G,p)$ is also rigid.
    By \Cref{t:dk} we have $\dim \col M(G,\phi_p) = d|V|-d$.
    Since $\dim \col M(G,\phi) = \rank M(G,\phi) \leq d|V|-d$,
    it follows that $\col M(G,\phi_p) = \col M(G,\phi)$.
\end{proof}

\begin{lemma}\label{l:generic2}
    Let $(G,p)$ be a generic rigid framework in $(\mathbb{R}^d,\|\cdot\|_\mathcal{P})$.
    Let $\Phi \subset \Phi_G$ be the set of directed colourings where $\col M(G,\phi_p) \subseteq \col M(G,\phi)$.
    If $q$ is a realisation such that $M(G,\phi)q = f_G(p)$ for some $\phi \in \Phi$,
    then $(G,q)$ is well-positioned.
\end{lemma}

\begin{proof}
	By applying translations, we may suppose that $q(u) = p(u) = \mathbf{0}$ for some vertex $u \in V$. 
	We note here that this implies that the $d|V|-d$ coordinates of $p$ with the vertex $u$ removed are distinct and algebraically independent over $\mathbb{Q}(\mathcal{P})$.
	We now fix $Z := \{ \tilde{p} \in (\mathbb{R}^d)^V : \tilde{p}(u) = \mathbf{0}\}$,
	and observe that $p$ is a generic point of $Z$.
	
	By \Cref{l:generic},
	$\rank M(G,\phi) = \rank M(G,\phi_p) = d|V|-d$ for each $\phi \in \Phi$.
	Hence,
	$\ker M(G,\phi) = \{(x_v)_{v \in V} : x \in \mathbb{R}^d\}$ for each $\phi \in \Phi$.
	Fix $X \subset Z$ to be the set of not well-positioned realisations of $G$.
	The set $X$ is a union of linear spaces $X_1,\ldots,X_n$ of codimension 1 or greater contained in $Z$,
	each defined over $\mathbb{Q}(\mathcal{P})$.
	For each $\phi \in \Phi$,
	fix the invertible linear map
	\begin{equation*}
		f_\phi : Z \rightarrow \col M(G,\phi_p), ~ \tilde{p} \mapsto M(G,\phi)\tilde{p}.
	\end{equation*}
	Importantly, every $f_\phi$ map is a $\mathbb{Q}(\mathcal{P})$-rational linear map.
	As $p$ is a generic point of $Z$ and $f_{\phi_p}$ is an invertible linear map,
	$f_G(p) = f_{\phi_p}(p)$ is a generic point in $\col M(G,\phi_p)$;
	specifically, $f_G(p)$ does not satisfy any additional $\mathbb{Q}(\mathcal{P})$-rational polynomial constraints except what is required to be a point of $\col M(G,\phi_p)$.
	As each linear space $f_\phi(X_i)$ (with $\phi \in \Phi$, $i \in [n]$) is a linear subspace of $\col M(G,\phi_p)$ with codimension at least 1 defined by $\mathbb{Q}(\mathcal{P})$-rational linear constraints,
	we have that $f_G(p) \notin f_\phi(X_i)$.
	Hence, $q \notin X$ as required.
\end{proof}

\Cref{l:generic,l:generic2} now gifts us information about all frameworks equivalent to a generic framework.

\begin{proposition}\label{p:generic}
    Let $(G,p)$ be a generic framework in $(\mathbb{R}^d,\|\cdot\|_\mathcal{P})$.
    If $(G,q)$ is equivalent to $(G,p)$ then the following properties hold.
    \begin{enumerate}
    	\item If $(G,p)$ is rigid, then $(G,q)$ is well-positioned and infinitesimally rigid.
    	\item If $(G,p)$ is redundantly rigid, then $(G,q)$ is redundantly rigid.
    \end{enumerate}
\end{proposition}

\begin{proof}       
    Since $(G,p)$ and $(G,q)$ are equivalent,
    the equality 
    \begin{equation*}
    		M(G,\phi)q = f_G(q) = f_G(p)    
    \end{equation*}    
    holds for each $\phi \in \Phi_q$,
    and so $f_G(p) \in \col M(G,\phi)$ for each $\phi \in \Phi_q$.
    Hence by \Cref{l:generic},
    $\rank M(G,\phi) = d|V|-d$ for all $\phi \in \Phi_q$.
    Furthermore,
    $(G,q)$ is well-positioned by \Cref{l:generic2}.
    Thus $(G,q)$ is infinitesimally rigid by \Cref{t:dk}.
    
    Now suppose $(G,p)$ is redundantly rigid.
    For each $e \in E$,
    the framework $(G-e,p)$ is generic and rigid,
    and so the framework $(G-e,q)$ is also infinitesimally rigid as was previously shown.
    Hence $(G,q)$ is redundantly rigid as required.
\end{proof}

We must stress here that, while genericity does simplify some of the conditions associated with rigidity in polyhedral normed spaces,
it is nowhere near as powerful a property as it is in Euclidean spaces.
One such weakness is presented more formally in the following result.

\begin{proposition}\label{prop:genericflexible}
	Let $G=(V,E)$ be a graph with $|V| \geq 2$.
	Then for any polyhedral normed space $(\mathbb{R}^d,\| \cdot \|_{\mathcal{P}})$ with $d \geq 2$,
	there exists a non-empty open subset $U \subset (\mathbb{R}^d)^V$ of realisations $p$ of $G$ where the framework $(G,p)$ is well-positioned and flexible in $(\mathbb{R}^d,\| \cdot \|_{\mathcal{P}})$.
	In particular,
	there exists a generic framework $(G,p)$ in $(\mathbb{R}^d,\| \cdot \|_{\mathcal{P}})$ that is flexible.
\end{proposition}

\begin{proof}
	Choose a smooth point $x \in \mathbb{R}^d$ and label the $n=|V|$ vertices of $G$ by $v_1,\ldots,v_n$.
	Let $x^*$ be the support vector of $x$.
	As $(\mathbb{R}^d,\| \cdot \|_{\mathcal{P}})$ is a polyhedral normed space,
	there exists $\varepsilon > 0$ such that for every $y \in \mathbb{R}^d$ where $\|x-y\|_\mathcal{P} < \varepsilon$,
	$y$ is smooth with support vector $x^*$.
	Fix $p:V \rightarrow \mathbb{R}^d$ by setting $p(v_j) = j x$ for each $j \in \{1,\ldots,n\}$.
	Using $p$,
	we define the following open set:
	\begin{equation*}
		U : = \left\{ q \in (\mathbb{R}^d)^V : \|p(v) - q(v)\|_{\mathcal{P}} < \varepsilon/2 \text{ for all } v \in V \right\}.
	\end{equation*}
	We now note that for any $q \in U$ and any edge $v_i v_j \in E$ ($i<j$),
	we have
	\begin{align*}
		\left\| x - \frac{q(v_j) - q(v_i)}{j-i} \right\|_{\mathcal{P}} &= 	\frac{\|(j-i) x - (q(v_j) - q(v_i))\|_{\mathcal{P}} }{j-i}\\
		&= \frac{\|(p(v_j) - q(v_j)) - (p(v_i) -q(v_i))\|_{\mathcal{P}}}{j-i} \\
		&\leq \frac{\|p(v_j) - q(v_j)\|_{\mathcal{P}} + \|p(v_i) -q(v_i)\|_{\mathcal{P}}}{j-i} \\
		&< \frac{(\varepsilon/2)+ (\varepsilon/2)}{j-i} \\
		& < \varepsilon,
	\end{align*}
	and hence, the support vector of $q(v_j) - q(v_i)$ is $x^*$.
	It follows that every $q \in U$ is a well-positioned realisation of $G$ with $\phi_q = \phi_p$.
	Since $\phi_p(vw) = x^*$ for every edge $vw \in E$,
	it is easy to check that $\rank M(G,\phi_p) \leq |V| - 1 < d|V| - d$.
	Hence, $(G,q)$ is flexible by \Cref{t:dk}.
\end{proof}

\subsection{\texorpdfstring{$\ell_\infty$}{L-infinity} normed spaces}\label{subsec:linf}

We conclude the section with some more specific results for each of the $\ell_\infty$ normed spaces.
For this subsection we now fix $b_1, \ldots , b_d$ to be the basis of $\mathbb{R}^d$ where the $j$-th coordinate of $b_i$ is 1 if $j=i$ and 0 otherwise.
With this,
the faces for $\ell_\infty^d$ are the vectors $\pm b_1,\ldots,\pm b_d$.

Fix an ordering on the vertex set of $G=(V,E)$ and choose any $\phi \in \Phi_G$.
We require two new functions;
the \emph{colouring} $|\phi| : E \rightarrow \{1,\ldots,d\}$ where $|\phi|(e) = i$ if and only if $\phi(e) \in \{b_i,-b_i\}$,
and the \emph{sign map} $\sgn(\phi) :E \rightarrow \{1,-1\}$ where $\sgn(\phi)(e) = \pm 1$ if and only if $\phi(e) = \pm b_i$ for some $i \in \{1,\ldots,d\}$.
It is immediate that $\phi(e) = \sgn(\phi)(e) b_{|\phi(e)|}$ for every edge $e \in E$.
With this we now define \emph{monochromatic subgraph} $G_i(\phi) = (V,E_i(\phi))$ where $E_i(\phi) := \{ e \in E : |\phi|(e) = i \}$.

\begin{theorem}[Kitson \cite{Kitson15}]\label{t:colour}
    Let $G=(V,E)$ be a graph with a well-positioned realisation $p$ in $\ell_\infty^d$.
    Then the framework $(G,p)$ is infinitesimally rigid if and only if each monochromatic subgraph $G_i(\phi_p)$ is connected.
\end{theorem}

The following is the first (as far as the author is aware) construction of a (well-positioned) infinitesimally rigid framework in $\ell_\infty^d$ for $d \geq 3$.
It was previously known that infinitesimally rigid frameworks could be found in higher-dimensional $\ell_\infty$ normed spaces by relaxing our requirement for such frameworks to be well-positioned;
see, for example, \cite[Example 3.13]{polyold}.

\begin{proposition}\label{p:k2d}
    For every $n \geq 2d$, the complete graph $K_{n}$ has an infinitesimally rigid realisation in $\ell_\infty^d$.
\end{proposition}

\begin{proof}
    Fix $G=(V,E)$ to be a graph that is isomorphic to $K_{2d}$ with $V = \{1,\ldots,d, -1, \ldots, -d\}$.
    Fix a scalar $0 < \varepsilon < 1/2$.
    Define $p:V \rightarrow \mathbb{R}^d$ to be the realisation where for each $1 \leq i \leq d$ we have $p(-i) = - p(i)$ and
    \begin{equation*}
        p(i) := ( ~ 0 ~ , ~ \ldots ~ ,  ~ 0  ~ ,  ~ \overbrace{1}^{i} ~ , ~ \varepsilon ~ , ~ \ldots ~ , ~ \varepsilon ~ ).
    \end{equation*}
    For every pair of integer $1 \leq i < j \leq d$,
    we note that
    \begin{align*}
        p(i) - p(j) &= ( ~ 0 ~ , ~ \ldots ~ , ~ 0 ~ , ~ \overbrace{1}^{i} ~ , ~ \varepsilon ~ , ~ \ldots ~ , ~ \varepsilon ~ , ~ \overbrace{-1 + \varepsilon}^{j} ~ , ~ 0 ~ , ~ \ldots ~ , ~ 0 ~ ), \\
        p(i) - p(-j) &= ( ~ 0 ~ , ~ \ldots ~ , ~ 0 ~ , ~ \overbrace{1}^{i} ~ , ~ \varepsilon ~ , ~ \ldots ~ , ~ \varepsilon ~ , ~ \overbrace{1 + \varepsilon}^{j} ~ , ~ 2\varepsilon ~ , ~ \ldots ~ , ~ 2\varepsilon ~ ).
    \end{align*}
    Furthermore, for each $1 \leq i \leq d$ we have $p(i) - p(-i) = 2p(i)$.
    Hence the framework $(G,p)$ is well-positioned in $\ell^d_\infty$, and for every pair $1 \leq i < j \leq d$ we have
    \begin{align*}
        |\phi_p|(ij) = |\phi_p|((-i)(-j)) = i,  \qquad |\phi_p|(i(-j)) = |\phi_p|((-i)j) = j,  \qquad |\phi_p|(i(-i))= i.
    \end{align*}
    
    Now fix an integer $1 \leq i \leq d$. 
    As seen above,
    $i$ is adjacent in $G_i(\phi_p)$ to every vertex $j$ with $j > i$ and every vertex $-k$ with $0 < k \leq i$.
    Similarly,
    $-i$ is adjacent in $G_i(\phi_p)$ to every vertex $-j$ with $j > i$ and every vertex $k$ with $0 < k \leq i$.
    Hence the graph $G_i(\phi_p)$ is connected.
    As this holds for each $1 \leq i \leq d$,
    the framework $(G,p)$ is infinitesimally rigid in $\ell_\infty^d$ by \Cref{t:colour}.
    
    Now fix $n > 2d$.
    Let $G' = (V',E')$ be the graph with $V' = V \cup \{d+1,\ldots,n-d\}$ and
    \begin{equation}
    	E' = E \cup \{ v w : 1 \leq v \leq d, ~ d+1 \leq w \leq n-d \}.
    \end{equation}
    We now define $q$ to be the well-positioned realisation of $G$ where $q(v) = p(v)$ if $v \in V$ and $q(v) = \mathbf{0}$ if $v \in V' \setminus V$.
    With this,
    we see that $G'_i(\phi_q) = G_i(\phi_p) + \{i j : d+1 \leq j \leq n-d \}$ for each $i \in [d]$,
    and so $(G',q)$ is infinitesimally rigid in $\ell_\infty^d$.
    From this it follows that $K_n$ is rigid in $\ell_\infty^d$.
\end{proof}

We end the section with the following example that illustrates some of the difficulties with working with $\ell_\infty$ normed spaces.
In particular, just because a graph is $(d,d)$-redundant does not imply that it is has a redundantly rigid realisation in $\ell_\infty^d$.

\begin{example}\label{ex:k5}
	By \Cref{t:dewar}, the graph $K_5-e$ has an infinitesimally rigid realisation in $\ell_\infty^2$ for every edge $e \in  E$.
However,
$K_5$ has no redundantly rigid realisation in $\ell_\infty^2$.
	To see this, let $G=(V,E)$ be a complete graph with vertex set $V = \{v_1,\ldots,v_5\}$ and let $(G,p)$ be an infinitesimally rigid framework in $\ell_\infty^2$.
	Suppose for contradiction that $(G,p)$ is also redundantly rigid.
	It follows from \Cref{t:colour} that the monochromatic subgraphs $G_1(\phi_p), G_2(\phi_p)$ are both 2-edge-connected with 5 edges, and so both are cycles.
	By relabelling the vertices of $G$ if necessary,
	we may suppose that the vertices of $G_1(\phi)$ are ordered so that $v_i v_{i+1} \in E_1(\phi_p)$ for each $i \in [5]$ (from here on we consider $v_6=v_1$ and $v_0 = v_5$).
	Now fix $D$ to be directed graph formed from $G_1(\phi)$ by directing the edge $v_iv_{i+1}$ as $v_i \rightarrow v_{i+1}$ if and only if $\phi_p(v_i v_{i+1}) = (1,0)$.
	We now observe that for $i \in [5]$, we cannot have either of the directed paths $v_{i-1} \rightarrow v_i \rightarrow v_{i+1}$ or $v_{i+1} \rightarrow v_i \rightarrow v_{i-1}$.
	For example, if the former were true, then we would have $\phi_p(v_{i-1} v_{i+1}) = (1,0)$ since
	\begin{align*}
		p_1(v_{i-1}) - p_1(v_{i+1}) &= \big(p_1(v_{i-1}) - p_1(v_{i}) \big) + \big(p_1(v_{i}) - p_1(v_{i+1}) \big) \\
		&> \big | p_2(v_{i-1}) - p_2(v_{i}) \big| + \big|p_2(v_{i}) - p_2(v_{i+1}) \big| \\
		&= \big | p_2(v_{i-1}) - p_2(v_{i+1}) \big | ,
	\end{align*}
	which contradicts that $v_{i-1}v_{i+1}$ is not an edge of $G_1(\phi_p)$.
	This implies $D$ is an alternating directed cycle, but this is impossible as it has an odd number of vertices.	
\end{example}

We observe here that \Cref{ex:k5} provides an answer in the negative to an open problem set by Clinch and Kitson \cite[Open problem 7]{ck20}.

\section{Equivalent conditions for global rigidity}\label{sec:equiv}

In this section we identify an exact condition for a well-positioned framework in a polyhedral normed space to be global rigidity.
After doing so, we explore the complexity limitations of this characterisation for non-generic frameworks.

\subsection{General polyhedral normed spaces}

We define two directed colourings $\phi$ and $\phi'$ to be \emph{isometric} if there exists a linear isometry $T$ of $(\mathbb{R}^d,\|\cdot\|_\mathcal{P})$ where $\phi = T \circ \phi' := (T(\phi'(e)))_{e \in E}$.
If $\phi \in \Phi_p$, then the isometric directed colouring $\phi' = T \circ \phi$ is contained in $\Phi_{p'}$ for the realisation $p'=T \circ p := (T(p(v)))_{v \in V}$;
this follows from the observation that $M(G,T \circ \phi)(T \circ p) = M(G,\phi)p$.
Using the language of isometric directed colourings,
we can restate what it means for a framework to be globally rigid in a polyhedral normed space.

\begin{proof}[Proof of \Cref{thm:compute}]
    First,
    suppose that $(G,p)$ is globally rigid.
    Choose any $\phi \in \Phi_G$ where $f_G(p) \in \col M(G,\phi)$.
    If $\phi$ is isometric then we are done.
    Suppose that $f_G(p)$ is contained in $f_G\left(M(G,\phi)^{-1}(f_G(p)) \right)$.
    Then there exists a realisation $q$ where $f_G(p) = f_G(q) = M(G,\phi)q$.
    Since $(G,p)$ is globally rigid,
    $(G,q)$ is congruent to $(G,q)$.
    As $\phi \in \Phi_q$,
    it follows that $\phi$ is isometric to $\phi_p$.
    Thus property \ref{thm:compute2} holds.    
    
    Now suppose that \ref{thm:compute2} holds.
    Let $(G,q)$ be another framework where $f_G(q) = f_G(p)$.
    Since $f_G(p) = f_G(q) = M(G,\phi)q$ for all $\phi \in \Phi_q$ (i.e., $f_G(p) \in \col M(G,\phi)$ for all $\phi \in \Phi_q$),
    every element of $\Phi_q$ is isometric to $\phi_p$.
    Choose any $\phi \in \Phi_q$ and let $T$ be a linear isometry of $(\mathbb{R}^d,\|\cdot\|_\mathcal{P})$ where $\phi_p = T \circ \phi$.
    Then 
    \begin{equation*}
    	M(G,\phi_p)p = f_G(p) = f_G( q)= M(G, \phi)( q) = M(G,T \circ \phi)(T \circ q)= M(G,\phi_p)(T \circ q),
    \end{equation*}
    and hence $(G,T \circ q)$ lies in the affine space $p + \ker M(G,\phi_p)$.
    As $(G,p)$ is rigid we have that $\ker M(G,\phi_p) = \{ (x)_{v \in V} : x \in \mathbb{R}^d\}$ by \Cref{t:dk},
    and hence $(G,T \circ q)$ is a translation of $(G,p)$.
    It now follows that $(G,q)$ is congruent to $(G,p)$.
    This now concludes the proof.
\end{proof}

As discussed in \Cref{subsec:main},
\Cref{thm:compute} provides an algorithm for determining whether or not a well-positioned framework $(G,p)$ is globally rigid in a polyhedral normed space $(\mathbb{R}^d,\|\cdot\|_\mathcal{P})$:
\begin{itemize}
	\item[1.] Determine $\phi_p$ and construct $M(G,\phi_p)$.
	\item[2.] Compute the rank of $M(G,\phi_p)$.
	If $\rank M(G,\phi_p) < d|V| - d$, terminate the algorithm with the answer ``not globally rigid''; otherwise, proceed.
	\item[3.] Fix a vertex $v_0 \in V$. Construct the set $\Phi_G' = \Phi_G \setminus \{\phi \text{ isometric to } \phi_p \}$ and use this to construct the set $A = \{ M'(G,\phi) : \phi \in \Phi_G'\}$, where the matrix $M'(G,\phi)$ is constructed from $M(G,\phi)$ by deleting the columns associated to the vertex $v_0$.
	\item[4.] Create the set $B = \emptyset$.
	For each matrix $M'(G,\phi) \in S$,
	compute the rank of $M'(G,\phi) $ and compute the rank of the augmented matrix $[M'(G,\phi) | f_G(p)]$.
	If the rank of $[M'(G,\phi) | f_G(p)]$ is greater than the rank of $M'(G,\phi)$, remove $M'(G,\phi)$ from $A$.
	If not,
	compute the affine solution space $X_\phi \subset (\mathbb{R}^d)^{V \setminus \{v_0\}}$ and append to the set $B$.
	\item[5.] For each $X_\phi \in B$,
	run the linear program
	\begin{mini*}|l|
		{q \in X_\phi}{0}
		{}{}
		\addConstraint{M'(G,\psi)q - f_G(p) \leq 0, ~ M'(G,\psi) \in A}
	\end{mini*}
	If this terminates with $0$ then terminate the algorithm with the answer ``not globally rigid''; otherwise, remove $X_\phi$ from $B$.
	\item[6.] If the set $B$ is empty, terminate the algorithm with the answer ``globally rigid''.	
\end{itemize}
Since linear programs and linear equation can be solved using deterministic algorithms which are subexponential in run time (assuming that the polytope $\mathcal{P}$ has rational faces), the above algorithm can be implemented as an exponential-time deterministic algorithm.
Hence one can always determine if any given well-positioned framework in a polyhedral normed space is globally rigid or not if given enough time.

The author stresses here that the above algorithm is in some sense a ``proof of concept'' algorithm, and no attempt has been made to optimise it.
In its current form,
the algorithm has an extremely long run time for even small examples due to the sheer size of $\Phi_G$.
For example, given a well-positioned framework with $m$ edges in $(\mathbb{R}^d, \| \cdot\|_\infty)$,
the set $\Phi_G$ contains $(2d)^m$ elements.
This implies that checking if a given realisation of $K_9$ in $(\mathbb{R}^4,\| \cdot\|_\infty)$ is globally rigid would require checking over $3.2 \times 10^{32}$ different $36 \times 36$ matrices, possibly checking each matrix multiple times;
to put this into perspective, if every computer on Earth (estimated to be roughly 3 billion) was ran in parallel, each checking one directed colouring per nanosecond, it would still take over 3.4 million years to determine if a framework $(K_9,p)$ was globally rigid in $\ell_\infty^4$.
What one should instead do is limit themselves to ``sensible'' directed colourings;
for example,
there is no need to check any directed colouring that is isometric to a previous directed colouring.
Some of these issues can be solved by requiring stronger sufficient conditions for global rigidity.
This idea is explored further in \Cref{sec:suff}.

An interesting alternative application of \Cref{thm:compute} is the following.

\begin{proposition}\label{p:singleexample}
	Let $(G,p)$ be a well-positioned and globally rigid framework in a polyhedral normed space $(\mathbb{R}^d,\|\cdot \|_\mathcal{P})$.
	Then there exists an open neighbourhood $U$ of $p$ such that every framework $(G,q)$ with $q \in U$ is globally rigid in $(\mathbb{R}^d,\|\cdot \|_\mathcal{P})$.	
\end{proposition}

\begin{proof}	
	For every realisation $q$ of $G$ in $(\mathbb{R}^d,\|\cdot \|_\mathcal{P})$,
	fix $S(q) \subset \Phi_G$ to be the set of directed colourings $\phi$ for which $f_G(q) \in \col M(G,\phi)$ but $\phi$ is not isometric to a directed colouring contained in $\Phi_q$.
	As the map $f_G$ is continuous and $p$ is well-positioned,
	it follows that there an open neighbourhood $U'$ of $p$ where for each $q \in U'$, the following holds: (i) $(G,q)$ is well-positioned, (ii) $\phi_q = \phi_p$, and (iii) $S(q) \subset S(p)$.
	By \Cref{thm:compute},
	we have that $f_G(p)$ is not contained in $f_G\left(M(G,\phi)^{-1}(f_G(p)) \right)$ for each $\phi \in S(p)$.
	
	Suppose for contradiction that the set $U$ as described in the statement of the result does not exist.
	Then there exists a sequence of realisations $(p_n)_{n \in \mathbb{N}}$ contained in $U'$ which converge to $p$ where each framework $(G,p_n)$ is not globally rigid.
	By applying \Cref{thm:compute} and replacing $(p_n)_{n \in \mathbb{N}}$ with one of its own subsequences if necessary,
	we may assume without loss of generality that there exists $\phi \in S(p)$ such that
	\begin{equation*}
		f_G(p_n) \in f_G\left(M(G,\phi)^{-1}\left(f_G(p_n)\right) \right).
	\end{equation*}	
	for all $n \in \mathbb{N}$.
	Since $G$ is connected,
	there exists a compact set $K \subset (\mathbb{R}^d)^V$ that contains a sequence $(q_n)_{n \in \mathbb{N}}$ such that $M(G,\phi)$ $f_G(p_n) = f_G(q_n) = M(G,\phi)(q_n)$ for all $n \in \mathbb{N}$.
	Using the compactness of $K$,
	we can replace $(p_n)_{n \in \mathbb{N}}$ with one of its subsequences so as to guarantee that $(q_n)_{n \in \mathbb{N}}$ converges to some realisation $q$.
	Hence, $f_G(p) = f_G(q) = M(G,\phi)(q)$.
	However,
	this now contradicts that $f_G(p)$ is not contained in $f_G\left(M(G,\phi)^{-1}(f_G(p)) \right)$.
	This contradiction now concludes the proof.
\end{proof}

Compare \Cref{p:singleexample} to global rigidity in a Euclidean space with dimension $d$.
Here finding a single well-positioned (i.e., no zero-length edges) globally rigid realisation is not sufficient for a graph have a non-empty open set of globally rigid realisations:
indeed, every cycle graph has a well-positioned globally rigid realisation in $d$-dimensional Euclidean space\footnote{Let $G$ be the cycle graph with vertices $\{1,\ldots,n\}$ and edges $\{12, \ldots, (n-1)n, n1\}$. Given the realisation $p:V \rightarrow \mathbb{R}^d$ with $p(i) = (i, 0, \ldots,0)$, the framework $(G,p)$ is globally rigid in $d$-dimensional Euclidean space.},
but every generic realisation of $G$ is not even rigid by \Cref{t:laman}.

The condition in \Cref{p:singleexample} that the framework is well-positioned is, however, required for the result.
Take for example the complete graph $K_4$.
As a consequence of \Cref{t:globness}, the graph $K_4$ has no globally rigid realisation in $\ell_\infty^2$.
However, the not well-positioned realisation $p: \{1,2,3,4\} \rightarrow \ell_\infty^2$ with
\begin{equation*}
	p(1) = (1,1), \qquad p(2) = (-1,1), \qquad p(3) = (1,-1), \qquad p(4) = (-1,-1)
\end{equation*}
is globally rigid in $\ell_\infty^2$ (see \cite[Theorem 4]{petty}).

\subsection{Exploits and limitations for \texorpdfstring{$\ell_\infty^d$}{L-infinity}}\label{subsec:linfnp}

Somewhat surprisingly,
our previous algorithms can be significantly improved for complete graphs in the case of $\ell_1^2$, and hence also the isometrically isomorphic space $\ell_\infty^2$.

\begin{theorem}[\cite{ccv11,epp11}]\label{thm:fastalgcomplete}
		There exists an $O(n^2)$ time algorithm that finds all the isometric embeddings of an $n$-point metric space in either $\ell_1^2$ or $\ell_\infty^2$.
		Hence,
		there exists an $O(n^2)$ time algorithm that determines if a framework $(K_n,p)$ in either $\ell_1^2$ or $\ell_\infty^2$ is globally rigid.
\end{theorem}

The analogue of \Cref{thm:fastalgcomplete} for $\ell_\infty^d$ with $d \geq 3$ is not true: specifically, determining if a $n$-point metric space can be isometrically embedded into $\ell_\infty^d$, $d \geq 3$, is NP-complete \cite[Theorem 2]{Edmonds2008}.
It is not clear if this necessarily implies determining if a complete framework is globally rigid in $\ell_\infty^d$ for $d \geq 3$ is NP-Hard.
Similar to \Cref{thm:fastalgcomplete},
there exists a polynomial-time algorithm for determining if a framework $(K_n,p)$ in $d$-dimensional Euclidean space is globally rigid for any positive integer $d$.
However, it was shown by Saxe \cite{saxe1980} that determining if a general framework is globally rigid in $d$-dimensional Euclidean space is NP-Hard for all $d \geq 1$.
It is likely that this is also true for any normed space.
We provide a proof below for $\ell_\infty$ spaces.
We first state the following useful result.

\begin{theorem}[Danzer and Gr{\"u}nbaum {\cite[II b$\alpha$]{Danzer1962}}]\label{t:danzer}
	Let $(G,p)$ be a framework in $\ell_\infty^d$ where $G=(V,E)$ is a complete graph.
	If $\|p(v)-p(w)\|_\infty = 1$ for every edge $vw \in E$,
	then $|V| \leq 2^d$.
	Furthermore,
	if $|V|=2^d$ then 
	\begin{equation*}
		\{p(v) : v \in V\} = \{-1,1\}^d + x
	\end{equation*}
	for some $x \in \mathbb{R}^d$,
	and so $(G,p)$ is globally rigid.
\end{theorem}

\begin{theorem}\label{thm:nphard}
	For every positive integer $d$,
	determining if a framework $(G,p)$ in $\ell_\infty^d$ is globally rigid is NP-Hard.
\end{theorem}

\begin{proof}
	The result was previously shown to be true when $d=1$ in \cite{saxe1980}.
	Because of this, we are now going to prove the following sufficient result:
	if $(G,p)$ is a framework in $\ell_{\infty}^{1} = (\mathbb{R},|\cdot|)$ where $G$ is connected,
	then there exists a framework $(G',p')$ in $\ell_{\infty}^{d}$ which is globally rigid if and only if $(G,p)$ is globally rigid.
	
	Fix a framework $(G,p)$ in $\ell_{\infty}^{1}$ with a connected graph $G$.
	To avoid trivial cases, we may suppose that $(G,p)$ has an edge $v_0v_1$ with non-zero length.
	By applying translations and reflections,
	we may assume that $p(v_0) = 1$ and $p(v_1) = 1 + \lambda$ for some $\lambda > 0$.
	By scaling $(G,p)$ about the point 1 by a positive scalar if necessary,
	we also may assume that $|p(v) - p(v_0)|  = |p(v) - 1| < \frac{1}{2|E|}$ for each $v \in V$.
	Now fix the sets
	\begin{equation*}
		S := \{-1,1\}^d, \qquad S_+ := \{ x \in S : x_1 = 1\}.
	\end{equation*}
	Set $G'=(V',E')$ to be the graph with $V' := V \cup S$ and
	\begin{equation*}
		E' := E \cup \{xy : x,y \in S ,~ x \neq y\} \cup \{v_0x, v_1 x : x \in S\} \cup \{ v x : v \in V, ~ x \in S_+\}.
	\end{equation*}
	Define $p'$ to be the realisation of $G'$ where $p'(v) = (p(v),\mathbf{0})$ for each $v\in V$ and $p'(x) = x$ for each $x \in S$.
	(See \Cref{fig:nphard} for any example of how $(G',p')$ is formed when $d=2$ and $G$ is the path $(v_0,v_1,v_2)$.)
	We observe here that the restriction of $(G',p')$ to the set of vertices $S$ is globally rigid by \Cref{t:danzer}.
	
	First suppose that $(G,p)$ is globally rigid.
	Let $(G',q')$ be a framework in $\ell_\infty^d$ that is equivalent to $(G',p')$.
	As the restriction of $(G',p')$ to the set of vertices $S$ is globally rigid,
	we may suppose (by applying isometries of $\ell_\infty^d$ if necessary) that $q'(x) = x$ for each $x \in S$.
	Choose any $v \in V$.
	For any vertex $v \in V$ and any $x \in S_+$,
	we have 
	\begin{equation}\label{eq:nphard}
		\max \Big\{ |q'_i(v) - x_i| : i \in [d] \Big\} = \|q'(v) - x \|_\infty = \|p'(v) - x \|_\infty = \max \Big\{ |p(v) - 1| , 1 \Big\} = 1.
	\end{equation}
	If there exist $i\in \{2,\ldots,d\}$ such that $q'_i(v) > 0$,
	then $q'_i(v) + 1 > 1$.
	However, this now contradicts \cref{eq:nphard}: just choose $x \in S_+$ such that $x_i =-1$.
	Following a similar reason, we see that $q'_i(v)$ cannot be negative either,
	hence $q'_i(v) = 0$ for each $i \geq 2$ and each $v \in V$.
	Furthermore,
	for each $v \in V$ we also have that $q'_1(v) \in [0,2]$:
	indeed if this were not true, we would have
	\begin{equation*}
		\|q'(v) - x\|_\infty = \max \Big\{ |q'_1(v) - 1|, 1\Big\} > 1 \qquad \text{ for any } x \in S_+,
	\end{equation*}
	which would contradict that $(G',q')$ is equivalent to $(G',p')$.
	
	Using that $q'_i(v_0) = 0$ for $i \geq 2$ and $q'_1(v_0) \in [0,2]$,
	it follows that for each edge $v_0 x$ with $x \in S \setminus S_+$,
	we have
	\begin{equation*}
		\max \Big\{ q'_1(v_0) +1 , 1 \Big\} = \|q'(v_0)-x\|_\infty = \|p'(v_0)-x\|_\infty = 2,
	\end{equation*}
	and hence $q'(v_0 ) = (1,0,\ldots,0)$.
	There are two possible positions for $q'(v_1)$:
	$(1 + \lambda, \mathbf{0})$ or $(1 + \lambda, \mathbf{0})$.
	If the latter were to hold then using the additional edge constraints $v_1 x$ for each $x \in S \setminus S_+$ and the observation that $0 < \lambda < \frac{1}{|E|} < 1$,
	we would see that
	\begin{equation*}
		2 - \lambda = \max \Big\{ (1 - \lambda) +1 , 1 \Big\} = \|q'(v_1) - x \|_\infty = \|p'(v_1)-x\|_\infty = 2 + \lambda,
	\end{equation*}
	which is a contradiction.
	Hence,
	$q'(v_1 ) = (1 + \lambda,0,\ldots,0)$.
	Now let $q$ be the realisation of $G$ in $\ell_\infty^1$ where $q(v) = q'_1(v)$ for each $v \in V$.
	Then $(G,p)$ and $(G,q)$ are equivalent.
	As $(G,p)$ is globally rigid, $q(v_0) = 1$ and $q(v_1)  = 1 + \lambda$,
	we have that $q = p$.
	Hence $q'=p'$,
	which implies $(G',p')$ is globally rigid.
	
	Now suppose that $(G',p')$ is globally rigid.
	Choose a framework $(G,q)$ in $\ell_\infty^1$ which is equivalent to $(G,p)$.
	By applying translations and reflections if needed,
	we may suppose that $q(v_0) = 1$ and $q(v_1) = 1 + \lambda$.
	Choose any vertex $v \in V$.
	Given $v_0 = a_0, a_1 , \ldots , a_{n-1},a_n = v$ is a path from $v_0$ to $v$ with no repeated vertices (and hence $n \leq |E|$),
	we see that
	\begin{align*}
		|q(v) - q(v_0)| &\leq \sum_{i=0}^{n-1} |q(a_{i+1})-q(a_i)| \leq \sum_{i=0}^{n-1} \Big( |q(a_{i+1})-q(v_0)| + |q(a_{i})-q(v_0)| \Big) \\
		 &= \sum_{i=0}^{n-1} \Big( |p(a_{i+1})-p(v_0)| + |p(a_{i})-p(v_0)| \Big)< \sum_{i=0}^{n-1}\frac{1}{|E|}  \leq 1.
	\end{align*}
	Hence, $q(v) \in (0,2)$ for each $v \in V$.
	We now fix $q'$ to be the realisation of $G'$ in $\ell_\infty^d$ where $q'(v) = (q(v),0,\ldots,0)$ for each $v \in V$ and $q'(x) = x$ for each $x \in S$.
	Since $q(v) \in (0,2)$ for each $v \in V$,
	we have for any $x \in S_+$ that
	\begin{equation*}
		\|q'(v) - x\|_\infty = \max\Big\{ |q(v) - 1| , 1 \Big\} = 1 = \|p'(v) - x\|_\infty.
	\end{equation*}
	By checking the remaining edge constraint equations,
	we see that $(G',q')$ is equivalent to $(G',p')$.
	Since the latter is globally rigid,
	we have $q' = p'$.
	Hence $q=p$ and $(G,p)$ is globally rigid.
	This now concludes the proof.
\end{proof}

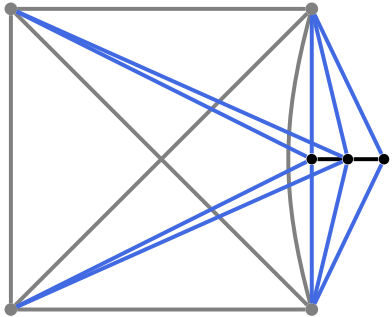
\begin{figure}[htp]
	\begin{center}
        \begin{tikzpicture}[scale=2]
			\node[hollow] (11) at (1,1) {};
			\node[hollow] (1-1) at (1,-1) {};
			\node[hollow] (-11) at (-1,1) {};
			\node[hollow] (-1-1) at (-1,-1) {};

			\draw[gedge] (11) to [bend right = 15] (1-1);
			\draw[gedge] (11)edge(-11);
			\draw[gedge] (11)edge(-1-1);
			\draw[gedge] (1-1)edge(-11);
			\draw[gedge] (1-1)edge(-1-1);
			\draw[gedge] (-11)edge(-1-1);

			\node[vertex] (v0) at (1,0) {};
			\node[vertex] (v1) at (1.24,0) {};
			\node[vertex] (v2) at (1.48,0) {};
			
			\draw[edge] (v0)edge(v1);
			\draw[edge] (v1)edge(v2);

			\draw[bedge] (v0)edge(11);
			\draw[bedge] (v0)edge(1-1);
			\draw[bedge] (v0)edge(-11);
			\draw[bedge] (v0)edge(-1-1);
			
			\draw[bedge] (v1)edge(11);
			\draw[bedge] (v1)edge(1-1);
			\draw[bedge] (v1)edge(-11);
			\draw[bedge] (v1)edge(-1-1);
			
			\draw[bedge] (v2)edge(11);
			\draw[bedge] (v2)edge(1-1);
		\end{tikzpicture}
	\end{center}
	\caption{Construction described in \Cref{thm:nphard} when $d=2$ and $G$ is the graph with vertices $v_0,v_1,v_2$ and edges $v_0v_1,v_1v_2$. The original framework $(G,p)$ is given in black (vertices from left to right: $v_0,v_1,v_2$), and the subframework of $(G',p')$ induced on $S$ is given in grey. Edges between vertices in $V$ and $S$ are given in blue.}\label{fig:nphard}
\end{figure}

\section{Necessary conditions for global rigidity}\label{sec:ness}

In this section we prove the necessary condition for global rigidity in polyhedral normed spaces that is described in \Cref{t:globness}.
The rough approach to proving \Cref{t:globness} is to approximate each polyhedral normed space by a sequence of normed spaces of the same dimension where global rigidity is well understood.
Specifically, we will approximate our polyhedral normed spaces using normed spaces that are \emph{analytic}: any normed space $(\mathbb{R}^d,\|\cdot\|)$ where the restriction of the norm $\|\cdot\|$ to the set $\mathbb{R}^d \setminus \{0\}$ forms a real analytic function.
An important family of analytic normed spaces are the $\ell_q$ normed spaces when $q$ is an even integer.
Using our approximation methods, we prove that any globally and infinitesimally rigid framework in our fixed polyhedral normed space must also be globally and infinitesimally rigid in some analytic normed space.
From this, we can then use various known properties regarding global rigidity in analytic normed spaces to prove \Cref{t:globness}.

%\subsection{Global rigidity in analytic normed spaces}

Before we begin, we require a few specific results regarding analytic normed spaces.
The first is a combination of multiple results from \cite{DN22} and \cite{DHN22}),
as well as \Cref{l:km}.

\begin{theorem}\label{t:analyticsuff}
	Let $(G,p)$ be an infinitesimally rigid and globally rigid framework in an analytic normed space $(\mathbb{R}^d,\|\cdot\|)$ with finitely many linear isometries.
	Then $G$ is $\mathcal{M}(d,d)$-connected.
\end{theorem}

\begin{proof}
	By \cite[Theorem 3.10]{DN22},
	there exists a non-empty open set of globally rigid realisations of $G$ in $\mathbb{R}^d$.
	By \cite[Proposition 3.6]{DN22} and \cite[Theorem 3.8]{DHN22},
	one such realisation $(G,q)$ is redundantly rigid in $(\mathbb{R}^d,\|\cdot\|)$,
	and by \cite[Theorem 3.1]{DHN22} the graph $G$ is 2-connected.
	The result now follows from \Cref{l:km}.
\end{proof}

%The next result we require is regarding the infinitesimally rigid realisations of a graph.

\begin{proposition}[{\cite[Proposition 3.2]{DN22}}]\label{p:regular}
	For any graph $G=(V,E)$ and any analytic normed space $(\mathbb{R}^d,\|\cdot\|)$,
	there exists an open conull\footnote{A subset of a linear space $\mathbb{R}^n$ is \emph{null} if it has Lebesgue measure 0, and \emph{conull} if its complement is a null set. Importantly, all conull subsets are dense.} set $U \subset (\mathbb{R}^d)^V$ so that the rank of $df_G(p)$ if maximal for all $p \in U$.
\end{proposition}

\begin{theorem}[see {\cite[Theorem 25.7]{rockafellar}}]\label{thm:rockafellar}
	Let $U \subset \mathbb{R}^d$ be an open convex set and let $(f_n)_{n \in \mathbb{N}}$ be a sequence of convex functions $f_n:U \rightarrow \mathbb{R}^d$ with pointwise limit $f:U \rightarrow \mathbb{R}^d$.
	Suppose that $f$ is differentiable, and each $f_n$ is also differentiable.
	Then $f_n \rightarrow f$ and $df_n \rightarrow  df$ uniformly on any compact set contained in $U$;
	specifically, for any compact set $K$ contained in $U$,
	we have
	\begin{equation*}
		\sup_{x \in K} |f(x) -f_n(x)| \rightarrow 0, \qquad \sup_{x \in K} \|df(x) - df_n(x) \| \rightarrow 0
	\end{equation*}
	as $n \rightarrow \infty$.
\end{theorem}

We also require the following small technical lemma.

\begin{lemma}\label{lem:expsums}
	Let $\mathcal{P}_1,\mathcal{P}_2 \subset \mathbb{R}^d$ be polytopes with faces $\mathcal{F}_1,\mathcal{F}_2$ respectively.
	For each $i=1,2$, let $h_i: \mathbb{R}^d \rightarrow \mathbb{R}$ be the map with $h_i(x) = \sum_{f \in \mathcal{F}_i} e^{f \cdot x}$.
	Then $h_1=h_2$ if and only if $\mathcal{P}_1=\mathcal{P}_2$.
\end{lemma}

\begin{proof}
	For this proof we use $f(t) \sim g(t)$ to denote the asymptotic equivalence of two functions $f,g:[0,\infty) \rightarrow \mathbb{R}$.
	Specifically, $f(t) \sim g(t)$ if and only if $\lim_{t \rightarrow \infty} f(t)/g(t) = 1$.

	Fix $X_i$ to be the well-positioned points in $\mathbb{R}^d$ with respect to $\|\cdot \|_{\mathcal{P}_i}$ and fix $\varphi_i$ to be the dual map of  $\|\cdot \|_{\mathcal{P}_i}$.
	For any $x \in X_1 \cap X_2$ and $t>0$, we have
	\begin{equation*}
		h_i(t x) = \sum_{f \in \mathcal{F}_i} e^{t(f \cdot x)} \sim e^{t(\varphi_i(x) \cdot x)} = e^{t \|x\|_{\mathcal{P}_i}}
	\end{equation*}
	for each $i=1,2$.
	Hence, $h_1(tx) \sim h_2(tx)$ for every $x \in X_1 \cap X_2$ if and only if $\|\cdot\|_{\mathcal{P}_1}=\|\cdot\|_{\mathcal{P}_2}$.
	The latter condition is equivalent to $\mathcal{P}_1= \mathcal{P}_2$,
	hence $h_1(tx) \sim h_2(tx)$ for every $x \in X_1 \cap X_2$ if and only if $\mathcal{P}_1= \mathcal{P}_2$.
	It is immediate that $h_1=h_2$ implies $h_1(tx) \sim h_2(tx)$ for every $x \in X_1 \cap X_2$ and $\mathcal{P}_1= \mathcal{P}_2$ implies $h_1=h_2$,
	and so $h_1=h_2$ if and only if $\mathcal{P}_1=\mathcal{P}_2$.
\end{proof}

Our next result tells us that we can always approximate a norm via a sequence of analytic norms.

\begin{lemma}\label{lem:convergeanalyticnorms}
	Let $(\mathbb{R}^d,\| \cdot\|_{\mathcal{P}})$ be a polyhedral norm.
	Then there exists a sequence of analytic norms $(\|\cdot\|_n)_{n \in \mathbb{N}}$ that uniformly converge to $\| \cdot\|$ such that each normed space $(\mathbb{R}^d,\| \cdot\|_{n})$ has the same isometries as $(\mathbb{R}^d,\| \cdot\|_{\mathcal{P}})$.
	Furthermore,
	if $\varphi$ is the dual map with respect to $\|\cdot\|_{\mathcal{P}}$ and $\varphi_n$ is the dual map with respect to $\|\cdot\|_n$,
	then $\varphi_n(x) \rightarrow \varphi(x)$ as $n \rightarrow \infty$ for each point $x \in \smooth (\mathbb{R}^d,\| \cdot\|_{\mathcal{P}})$.
\end{lemma}

\begin{proof}
	Recall that $\mathcal{F}$ is the set of faces of $\mathcal{P}$.	
	For any integer $n > \log |\mathcal{F}|$ (where $\log$ has base $e$),
	we define the analytic convex function
	\begin{equation*}
		\phi_{n}: \mathbb{R}^d \rightarrow \mathbb{R}, ~x \mapsto \frac{1}{|\mathcal{F}|} \sum_{f \in \mathcal{F}} e^{n(f\cdot x-1)}.
	\end{equation*}
%	with derivative
%	\begin{equation*}
%		d \phi_n(x) := \frac{n}{|\mathcal{F}|} \sum_{f \in \mathcal{F}} f e^{n(f\cdot x-1)} .
%	\end{equation*}
	Since the set $\{x \in \mathbb{R}^d : \phi_n(x) \leq 1\}$ is centrally symmetric and convex,
	each map
	\begin{equation*}
		x \mapsto \|x\|_n := \sup \{ t \geq 0 : \phi_n(x/t) \leq 1\}
	\end{equation*}
	is a norm of $\mathbb{R}^d$.
	It was shown in \cite[Lemma 4.10]{Dewar2022} that each norm $\| \cdot\|_n$ is analytic.
	
	Let $L: \mathbb{R}^d \rightarrow \mathbb{R}^d$ be a linear map.
	If $L$ is an isometry of $(\mathbb{R}^d,\| \cdot\|_{\mathcal{P}})$,
	then $L^T \mathcal{F} = \mathcal{F}$;
	hence, for any $n$ we have
	\begin{align*}
		\phi_n(Lx) &= \frac{1}{|\mathcal{F}|} \sum_{f \in \mathcal{F}} e^{n(f\cdot Lx-1)} = \frac{1}{|\mathcal{F}|} \sum_{f \in \mathcal{F}} e^{n(L^T f\cdot x-1)} \\
		&=  \frac{1}{|\mathcal{F}|} \sum_{g \in L^T\mathcal{F}} e^{n(g\cdot x-1)} =  \frac{1}{|\mathcal{F}|} \sum_{g \in \mathcal{F}} e^{n(g\cdot x-1)} = \phi_n(x),
	\end{align*}
	and so any linear isometry of $(\mathbb{R}^d,\| \cdot\|_{\mathcal{P}})$ is a linear isometry of each normed space $(\mathbb{R}^d,\| \cdot\|_{n})$.
	Now suppose that $L$ is not a linear isometry of $(\mathbb{R}^d,\| \cdot\|_{\mathcal{P}})$.
	Then $L^T \mathcal{F} \neq \mathcal{F}$.
	As 
	\begin{align*}
		\phi_n(Lx) \neq  \phi_n(x) \quad \iff \quad  \sum_{f \in \mathcal{F}} e^{nf\cdot Lx} \neq \sum_{f \in \mathcal{F}} e^{nf\cdot x},
	\end{align*}
	it follows from \Cref{lem:expsums} that $\phi_n \circ L \neq  \phi_n$.
	Hence, $L$ is not a linear isometry of $(\mathbb{R}^d,\| \cdot\|_n)$.
	By the Mazur-Ulam theorem,
	all isometries of a normed space are the composition of a linear isometry and a translation.
	Hence, 
	each normed space $(\mathbb{R}^d,\| \cdot\|_{n})$ has the same linear isometries as $(\mathbb{R}^d,\| \cdot\|_{\mathcal{P}})$.
	
	Since $\smooth (\mathbb{R}^d,\| \cdot\|_{\mathcal{P}})$ is the disjoint union of finitely many open convex sets,
	the final part of the result now follows from \Cref{thm:rockafellar}.
\end{proof}

Our next result shows how global rigidity can be preserved under converging norms.

\begin{lemma}\label{l:limit}
	Let $(\|\cdot\|_{n})_{n \in \mathbb{N}}$ be a sequence of smooth norms for $\mathbb{R}^d$ which uniformly converge to a norm $\|\cdot\|$.
	Further suppose that $(\mathbb{R}^d,\|\cdot\|)$ has finitely many linear isometries,
	and the isometries of $(\mathbb{R}^d,\|\cdot\|)$ and each space $(\mathbb{R}^d,\|\cdot\|_n)$ are the same.
	Then the following holds: 
	if $(G,p)$ is an infinitesimally rigid and globally rigid framework in $(\mathbb{R}^d,\|\cdot\|)$ and the set $\{p(v) :v \in V\}$ affinely spans $\mathbb{R}^d$,
	then $(G,p)$ is infinitesimally rigid and globally rigid framework in $(\mathbb{R}^d,\|\cdot\|_n)$ for all sufficiently large $n$.
\end{lemma}

\begin{proof}
	By applying translations to $(G,p)$, we may assume that $p(v_0) = \mathbf{0}$ for a fixed vertex $v_0$.
	Let $S$ be the finite set of linear isometries of $(\mathbb{R}^d,\|\cdot\|)$ (and hence the set of linear isometries of $(\mathbb{R}^d,\|\cdot\|_n)$ for each $n$).
	Let $Z = \{q \in (\mathbb{R}^d)^V : q(v_0) = \mathbf{0}\}$.
	Define the maps
	\begin{align*}
		h &: Z \rightarrow \mathbb{R}^E, ~  q \mapsto \left( \|q(v) - q(w)\| \right)_{vw \in E}, \\
		h_n &: Z \rightarrow \mathbb{R}^E, ~  q \mapsto \left( \|q(v) - q(w)\|_n \right)_{vw \in E}.
	\end{align*}
	Our set-up now implies the following:
	\begin{enumerate}
		\item as $(G,p)$ is infinitesimally rigid in $(\mathbb{R}^d,\|\cdot\|)$,
	the map $h$ is differentiable at $p$ with an injective derivative;
	\item as $(G,p)$ is globally rigid in $(\mathbb{R}^d,\|\cdot\|)$ and the set $\{p(v) :v \in V\}$ affinely spans $\mathbb{R}^d$,
	the fibre $h^{-1}(h(p))$ contains $|S|$ points, and those points are exactly the realisations $S \circ p := \{ f \circ p : f \in S\}$;
	\item $(G,p)$ is infinitesimally rigid in $(\mathbb{R}^d,\|\cdot\|_n)$ if and only if $h_n$ is differentiable at $p$ with an injective derivative;
	\item $(G,p)$ is globally rigid in $(\mathbb{R}^d,\|\cdot\|_n)$ if and only if the fibre $h_n^{-1}(h_n(p))$ contains $|S|$ points.
	\end{enumerate}
	By \Cref{thm:rockafellar},
	$dh_n(p) \rightarrow dh(p)$ entry-wise as $n \rightarrow \infty$.
	As $dh(p)$ is injective,
	it follows that $dh_n(p)$ is injective (and hence $(G,p)$ is infinitesimally rigid in $(\mathbb{R}^d,\|\cdot\|_n)$) for sufficiently large $n$.
	Moreover, the entry-wise convergence for $dh_n(p) \rightarrow dh(p)$ implies that there exists an open neighbourhood $U_{\text{inj}}$ of $p$ where $dh_n(q)$ is injective for all $q \in U_{\text{inj}}$ and all sufficiently large $n$.
	We note here that this implies each map $h_n|_{U_{\text{inj}}}$ is locally injective for sufficiently large $n$.
	
	As all of the normed spaces in question have the same isometries,
	we have that $S \circ p = h^{-1}(h(p)) \subseteq h_n^{-1}(h_n(p))$ for each $n \in \mathbb{N}$.
	
	\begin{claim}\label{claim:compact}
		There exists a compact set $K \subset Z$ such that $h^{-1}(h(p)) \subset K$ and $h_n^{-1}(h_n(p)) \subset K$ for each $n \in \mathbb{N}$.
	\end{claim}
	
	\begin{proof}
		Throughout the clam, we set $h_0 :=h$ and $\|\cdot\|_0:= \|\cdot\|$.
		
		We note that, as $h_n \rightarrow h$ uniformly as $n \rightarrow \infty$,
		there exists $C>0$ such that $\|p(v) - p(w)\|_n \leq C$ for all $n \in \mathbb{Z}_{\geq 0}$.
		For each vertex $v \in V$,
		fix a path $P_v = (v_0 v_1, v_1v_2, \ldots, v_{n-1}v_n)$ from $v_0$ to $v_n = v$,
		and fix $\ell$ to be the maximum length (i.e., number of edges) of the paths $\{P_v : v \in V\}$.
		Pick any non-negative integer $n \in \mathbb{Z}_{\geq 0}$ and choose any $q \in h_n^{-1}(h_n(p))$.
		Then we have that
		\begin{equation*}
			\max_{v \in V} \|q(v)\|_n = \max_{v \in V} \|q(v) - q(v_0)\|_n \leq \max_{v \in V} \sum_{xy \in P_v} \|q(x)-q(y)\|_n  = \max_{v \in V} \sum_{xy \in P_v} \|p(x)-p(y)\|_n \leq C \ell.
		\end{equation*}
		We now fix $K$ to be the set of all realisations $q \in Z$ where $\sup_{n \in \mathbb{Z}_{\geq 0}}\|q(v)\|_n \leq C \ell$ for each $v \in V$.
	\end{proof}		
	
	Now suppose for the sake of contradiction that there exists a subsequence $(\|\cdot\|_{n_k})_{k \in \mathbb{N}}$ so that each $h_{n_k}^{-1}(h_{n_k}(p))$ contains some realisation $q_{k} \notin S \circ p$.
	As each $h_{n_k}^{-1}(h_{n_k}(p))$ is contained in the compact set $K$,
	we can suppose without loss of generality that our subsequence was chosen such that $(q_k)_{n \in \mathbb{N}}$ converges.
	Since $\|\cdot\|_{n_k} \rightarrow \|\cdot\|$ uniformly, we have that $(q_k)_{n \in \mathbb{N}}$ converges to some point $f \circ p \in S \circ p$;
	moreover, by replacing each $q_k$ with $f^{-1} \circ q_k$, we can suppose that $q_k \rightarrow p$ as $k \rightarrow \infty$.
	
	As $h$ is differentiable at $p$ with injective derivative,
	it follows from the constant rank theorem (see, for example, \cite[Theorem 9.32]{rudin}) that there exists an open neighbourhood $U' \subset Z$ of $p$ such that $h|_{U'}$ is injective.
	Moreover, as $h_n \rightarrow h$ uniformly as $n \rightarrow \infty$ and each map $h_n|_{U_{\text{inj}}}$ is locally injective for sufficiently large $n$,
	there exists an integer $N$ and an open set $U \subset U'$ containing $p$ such that each map $h_n|_U$ is injective for $n \geq N$.
	However,
	for sufficiently large $k$ we have $n_k \geq N$, $q_k \in U$ and $h_{n_k}(q_k) = h_{n_k}(p)$.
	Since $q_k \neq p$ for each $k$,
	this provides the desired contradiction.
\end{proof}

We are now ready to prove our main theorem for the section.

\begin{proof}[Proof of \Cref{t:globness}]
	As $(G,p)$ is generic and globally rigid in $(\mathbb{R}^d,\|\cdot\|_{\mathcal{P}})$,
	then it is infinitesimally rigid in $(\mathbb{R}^d,\|\cdot\|_{\mathcal{P}})$ by \Cref{t:dk}.
    Let $(\|\cdot\|_n)_{n \in \mathbb{N}}$ be the sequence of norms guaranteed by \Cref{lem:convergeanalyticnorms}.
    As every polyhedral normed space has finitely many linear isometries,
    each normed space $(\mathbb{R}^d,\|\cdot\|_n)$ has finitely many linear isometries.
    By \Cref{l:limit},
    there exists an integer $N$ so that $(G,p)$ is infinitesimally rigid and globally rigid in $(\mathbb{R}^d,\|\cdot\|_n)$ for all $n \geq N$.
    The result now follows from \Cref{t:analyticsuff}.
\end{proof}

\section{Approximating \texorpdfstring{$\ell_\infty$}{L-infinity} global rigidity using \texorpdfstring{$\ell_p$}{Lp} global rigidity}\label{sec:linfplane}

In this section we prove \Cref{t:lpapprox}, which we then use to prove an exact characterisation of graphs with generic globally rigid realisations in $\ell_\infty^2$ (\Cref{t:linfplane}).

\subsection{Rigidity and global rigidity in \texorpdfstring{$\ell_p$}{Lp} spaces}

We first cover some background material needed for $\ell_p$ spaces.
As we usually save ``$p$'' to denote a realisation, we instead opt to use $\ell_q$ spaces from now on.
We recall that for any $1 \leq q < \infty$,
the $\ell_q$ norm for $\mathbb{R}^d$ is given by
\begin{equation*}
	\|(x_1,\ldots,x_d)\|_q := \left( \sum_{i=1}^d |x_i|^q \right)^{1/q},
\end{equation*}
and we set $\ell_q^d := (\mathbb{R}^d,\|\cdot\|_q)$.
For $1 < q < \infty$, the space $\ell_q^d$ is smooth.
When $q$ is an even integer, $\ell_q^d$ is an analytic normed space; moreover, since the $q$-th power of the norm (i.e., $\sum_{i=1}^d |x_i|^q$) is a polynomial with rational coefficients, every framework $(G,p)$ that is generic in $\mathbb{R}^d$ is regular, in that $\rank df_G(p) \geq \rank df_G(\tilde{p})$ for all other realisations $\tilde{p}$ in $\ell_q^d$ (see, for example, \cite[Proposition 3.1]{Connelly2005}).

Study of rigidity in $\ell_q^d$ was originally investigated by Kitson and Power \cite{KP14}.
Sugiyama and Tanigawa \cite{st24} recently investigated generic global rigidity in $\ell_q^d$, where they described a stress condition for generic global rigidity in $\ell_q^d$ for even $q \geq 4$ that is analogous to Connelly's stress condition for generic global rigidity in Euclidean spaces \cite{Connelly2005}.
Using this, they proved the following two results that we require for this section.

\begin{theorem}[{\cite[Theorem 5.8]{st24}}]\label{t:shin}
	Let $q$ be an even integer and let $(G,p)$ be a generic framework in $\mathbb{R}^2$.
	Then $(G,p)$ is globally rigid in $\ell_q^2$ if and only if $G$ is $\mathcal{M}(2,2)$-connected.
\end{theorem}

\begin{theorem}[{\cite[Theorem 2.11]{st24}}]\label{t:shin2}
	Let $q$ be an even integer and let $(G,p)$ be a generic framework in $\mathbb{R}^d$.
	If $G$ is 2-connected and has an infinitesimally rigid realisation in $\ell_q^{d+1}$, then $(G,p)$ is globally rigid in $\ell_q^d$.
\end{theorem}

\subsection{Approximating the \texorpdfstring{$\ell_\infty$}{L-infinity} norm via \texorpdfstring{$\ell_p$}{Lp} norms}

Our key lemma for proving \Cref{t:linfplane} is the following technical result.

\begin{lemma}\label{l:fmapdiff}
	Let $U \subset \mathbb{R}^d$ be the open dense set of points $x=(x_1,\ldots,x_d)$ where the maximum of the set $\{|x_1|, \ldots, |x_d|\}$ occurs exactly once.
	Then the map
	\begin{equation*}
		f: \mathbb{R}^d \times (-1,1)  \rightarrow \mathbb{R} , ~ (q,t) \mapsto
		\begin{cases}
			 \|x \|_{1/t}  &\text{if } t > 0, \\
			 \|x \|_{\infty} &\text{if } t \leq 0
		\end{cases}
	\end{equation*}
	is continuously differentiable on the set $U \times (-1,1)$.
	Moreover, the partial derivative of $f$ with respect to $t$ at $t=0$ is 0.
\end{lemma}

\begin{proof}
Every point $x \in U$ is a differentiable point for any $\ell_p$ norm, including $\ell_\infty$,
hence the partial derivative of $f$ with respect to $x$ exists and is continuous on the set $U \times (-1,1)$.

Fix a point $x \in U$.
It now remains to prove that the map $g: t \mapsto f(x, t)$ is continuous differentiable on the set $(-1,1)$ with $g'(0)=0$.
It is clear that if $t < 0$ then the derivative of $g$ at $t$ is 0.
If $t > 0$,
the derivative of $g$ at $t$ is
\begin{align*}
	g'(t) &= \frac{d}{dt} \left( \sum_{i=1}^d |x_i|^{1/t} \right)^t \\
	&=  \left( \sum_{i=1}^d |x_i|^{1/t} \right)^t \frac{d}{dt} t \log \left( \sum_{i=1}^d |x_i|^{1/t} \right) \\
	&= \|x\|_{1/t} \left(\log \left( \sum_{i=1}^d |x_i|^{1/t} \right) +   t \frac{d}{dt}\log \left( \sum_{i=1}^d |x_i|^{1/t} \right) \right) \\
	&= \|x\|_{1/t} \left(\log \left( \sum_{i=1}^d |x_i|^{1/t} \right) +   t \frac{\sum_{i=1}^d \frac{d}{dt} |x_i|^{1/t} }{\sum_{i=1}^d |x_i|^{1/t} }  \right) \\
	&= \|x\|_{1/t} \left(\log \left( \sum_{i=1}^d |x_i|^{1/t} \right) -    \frac{\sum_{i=1}^d  |x_i|^{1/t}  \log (|x_i|) }{t\sum_{i=1}^d |x_i|^{1/t} }  \right) \\
	&= \|x\|_{1/t} \left( \frac{\sum_{i=1}^d  |x_i|^{1/t} \left(t\log\left( \sum_{j=1}^d |x_j|^{1/t}\right) -  \log (|x_i|) \right)}{t\sum_{i=1}^d |x_i|^{1/t} }  \right) \\
	&= \|x\|_{1/t} \left( \frac{\sum_{i=1}^d  |x_i|^{1/t} \left(\log\left( \sum_{j=1}^d |x_j/x_i|^{1/t} \right) \right)}{\sum_{i=1}^d |x_i|^{1/t} }  \right),
\end{align*}
which is continuous on the interval $(0,1)$.
Furthermore,
it follows from the next claim that $g'(t) \rightarrow 0$ as $t \rightarrow 0+$.

\begin{claim}
For each $i \in \{1,\ldots,d\}$,
	\begin{equation*}
		\frac{|x_i|^{1/t} \log\left( \sum_{j=1}^d |x_j/x_i|^{1/t} \right) }{\sum_{j=1}^d |x_j|^{1/t} } \rightarrow 0 \qquad \text{ as } \qquad  t \rightarrow 0+.
	\end{equation*}
\end{claim}

\begin{proof}
	Since $x \in U$,
	there exists an index $\ell$ such that $|x_\ell| > |x_i|$ for each $i \neq \ell$.
	If $i \neq \ell$ then, since $|x_\ell/x_i| >1$,
	we have
	\begin{equation*}
		\frac{|x_i|^{1/t} \log\left( \sum_{j=1}^d |x_j/x_i|^{1/t} \right) }{\sum_{j=1}^d |x_j|^{1/t} } \leq \frac{|x_i|^{1/t} \log\left( d |x_\ell/x_i|^{1/t} \right) }{ |x_\ell|^{1/t} } = \frac{\log\left(  |x_\ell/x_i|^{1/t} \right) + \log(d) }{ |x_\ell/x_i|^{1/t} } \rightarrow 0
	\end{equation*}
	as $t \rightarrow 0+$.
	If $i=\ell$,
	then, since $|x_\ell/x_j| < 1$ for each $j \neq \ell$,
	we have
	\begin{equation*}
		\frac{|x_\ell|^{1/t} \log\left( \sum_{j=1}^d |x_j/x_\ell|^{1/t} \right) }{\sum_{j=1}^d |x_j|^{1/t} } \leq \frac{|x_\ell|^{1/t} \log\left( 1 + \sum_{j\neq \ell}^d |x_j/x_\ell|^{1/t} \right) }{ |x_\ell|^{1/t} } = \log\left( 1 + \sum_{j\neq \ell}^d |x_j/x_\ell|^{1/t} \right) \rightarrow 0
	\end{equation*}
	as $t \rightarrow 0+$.
\end{proof}

We now need to show that $g'(0) = 0$.
Fix the index $\ell$ so that $|x_\ell| > |x_i|$ for all $i \neq \ell$,
and fix the map $h:[0,1) \rightarrow \mathbb{R}$ defined by
\begin{align*}
	h(t) :=
	\begin{cases} 
		\sum_{i \neq \ell} |x_i/x_\ell|^{1/t} &\text{if } t > 0,\\
		0 &\text{if } t=0.
	\end{cases}
\end{align*} 
We now note three important properties for $h$:
(i) $h(t) \rightarrow 0$ as $t \rightarrow 0+$,
(ii) for $t>0$,
\begin{equation*}
	h'(t) = -\frac{\sum_{i \neq \ell} |x_i/x_\ell|^{1/t} \log \left(  |x_i/x_\ell| \right)}{t^2},
\end{equation*}
and
(iii) since $\lim_{y \rightarrow \infty} y/e^{\lambda y} = 0$ for any $\lambda >1$ (a consequence of L'H\^{o}pital's rule), we have
\begin{equation*}
	h'(0) = \lim_{t\rightarrow 0+}\frac{\sum_{i \neq \ell} |x_i/x_\ell|^{1/t} }{t} = \sum_{i \neq \ell}  \lim_{s\rightarrow \infty} \frac{s}{ e^{s \log(|x_\ell/x_i|)} }  = 0.
\end{equation*}
A similar application of L'H\^{o}pital's rule gives that $h'(t) \rightarrow 0$ as $t\rightarrow 0+$,
and so $h$ is a continuous differentiable map.
Using this,
we can prove the following claim.

\begin{claim}
	The continuous map
	\begin{equation*}
		\phi : [0,1) \rightarrow \mathbb{R}, ~ t \mapsto (1 + h(t))^t
	\end{equation*}
	is continuously differentiable with $\phi'(0)=0$.
\end{claim}

\begin{proof}
	Since $h$ is continuously differentiable and (by the chain rule) the composition of continuously differentiable functions is continuously differentiable, $\phi$ is continuously differentiable. We now see that
	\begin{equation*}
		\phi'(0) = \left( \phi(t) \frac{d}{dt} t \log (1+h(t)) \right)\Bigr|_{t=0} =
		\left( \phi(t)\left(\log (1+h(t)) + \frac{t h'(t)}{1+h(t)}\right)\right)\Bigr|_{t=0} 
		 =0.\qedhere
	\end{equation*}
\end{proof}

We now see that
\begin{align*}
	\lim_{t \rightarrow 0+} \frac{g(t) - g(0)}{t} = \lim_{t \rightarrow 0+} \frac{\left( \sum_{i=1}^d |x_i|^{1/t} \right)^t - |x_\ell|}{t}	= |x_\ell|\left(\lim_{t \rightarrow 0+} \frac{\phi(t) - \phi(0)}{t} \right)= |x_\ell| \phi'(0) = 0.
\end{align*}
Since $g(t)$ is constant on the interval $(-1,0]$, it is clear that $\frac{g(t) - g(0)}{t}\rightarrow 0$ as $t \rightarrow 0-$.
Hence, $g'(0) = 0$.
This now concludes the proof.
\end{proof}

\subsection{Proof of \texorpdfstring{\Cref{t:lpapprox,t:linfplane}}{two main theorems}}

The next lemma is a standard tool used when studying generic fibres of continuous maps.

\begin{lemma}\label{l:boundingfibres}
	Let $U \subset \mathbb{R}^m$ be an open set and let $f : U \rightarrow \mathbb{R}^n$ be a continuously differentiable map.
	Suppose that the following conditions hold for a point $x \in U$:
	\begin{enumerate}
		\item the set $f^{-1}(f(x))$ is finite;
		\item for every point $y\in f^{-1}(f(x))$, the derivative $df(y)$ is injective with $\ker df(y)^T = \ker df(x)^T$.
	\end{enumerate}
	Then there exists $\varepsilon >0$ such that the following holds: for all $z \in U$ with $\|x-z\| < \varepsilon$,
	the fibre $f^{-1}(f(z))$ contains at least as many elements as the fibre $f^{-1}(f(x))$.
\end{lemma}

\begin{proof}
	This is an immediate consequence of the constant rank theorem (see, for example, \cite[Theorem 9.32]{rudin}).
\end{proof}

We are now ready to prove \Cref{t:lpapprox}.

\begin{proof}[Proof of \Cref{t:lpapprox}]
	Throughout the proof, we fix $f_{G,q}$ to be the rigidity map with respect to $\ell_q^d$ for any $q \in [1,\infty]$.
	
	Fix a vertex $v_0 \in V$.
	Let $X \subset  (\mathbb{R}^d)^{V}$ be the affine space of all realisations $\tilde{p}$ where $\tilde{p}(v_0) = p(v_0)$,
	and let $Y$ be the open dense subset of well-positioned realisations (with respect to the ambient space $\ell_\infty^d$).
	Given $f$ is the map described in \Cref{l:fmapdiff},
	define the map
	\begin{equation*}
		F: X \times (-1,1)  \rightarrow \mathbb{R}^E \times (-1,1) , ~ (\tilde{p},t) \mapsto \left( \Big( f(\tilde{p}(v) - \tilde{p}(w), t) \Big)_{vw \in E}~ , ~ t \right).
	\end{equation*}
	By \Cref{l:fmapdiff},
	the map $F$ is continuously differentiable over the open dense set $Y \times (-1,1)$.
	Furthermore, the Jacobian of $F$ at any point $(\tilde{p},t) \in Y \times (-1,1)$ is the $(|E|+1) \times (d|V| +1)$ matrix
	\begin{equation*}
		dF(\tilde{p},t) =
		\begin{bmatrix}
			M_t(\tilde{p}) & a_t(\tilde{p}) \\
			\mathbf{0}^T & 1
		\end{bmatrix}
	\end{equation*}
	where:
	\begin{itemize}
		\item $M_t(\tilde{p})$ is formed from $df_{G,1/t}(\tilde{p})$ (or $df_{G,\infty}(\tilde{p})$ if $t=0$) by removing the $d$ columns associated to the vertex $v_0$,
		\item $\mathbf{0}$ is the all-zeroes vector of length $d|V|-d$, and
		\item $a_t(\tilde{p})$ is a vector of length $|E|$, with $a_t(\tilde{p})$ being the all-zeroes vector when $t = 0$.
	\end{itemize}
	Here we note that for any $\tilde{p} \in Y$ and any $t \in (-1,1)$,
	\begin{align*}
		\ker M_t(\tilde{p}) &=
		\begin{cases}
			\{ u \in \ker df_{G,1/t}(\tilde{p}) : u(v_0) = \mathbf{0}\} &\text{if } t >0\\
			\{ u \in \ker df_{G,\infty}(\tilde{p}) : u(v_0) = \mathbf{0}\}&\text{if } t =0,
		\end{cases}\\
		\ker M_t(\tilde{p})^T &=
		\begin{cases}
			\ker df_{G,1/t}(\tilde{p})^T &\text{if } t >0\\
			\ker df_{G,\infty}(\tilde{p})^T &\text{if } t =0.
		\end{cases}		
	\end{align*}
	Moreover, $\ker F(\tilde{p} ,0) = \ker M_0(\tilde{p}) \times \{0\}$ and $\ker F(\tilde{p} ,0)^T = \ker df_{G,\infty}(\tilde{p})^T \times \{0\}$.
	As $(G,p)$ is generic in $\ell_\infty^d$,
	it follows from \Cref{l:generic} that any framework $(G,\tilde{p})$ that is equivalent to $(G,p)$ in $\ell_\infty^d$ with $\tilde{p}(v_0) = p(v_0)$ satisfies $\ker dF(\tilde{p},0)^T = \ker dF(p,0)^T$.
	Since $(G,p)$ is infinitesimally rigid in $\ell_\infty^d$,
	we have that $dF(\tilde{p},0)$ is injective for each element of the fibre $F^{-1}(F(p,0))$.
	Using that $G$ is connected (a consequence of $(G,p)$ being rigid), it follows from a similar argument to that given in \Cref{claim:compact} that the fibre $F^{-1}(F(p,0))$ is contained in a compact set.
	Combining these last two facts implies that the fibre $F^{-1}(F(p,0))$ is finite.
	Hence, the pair $F$ and $(p,0)$ satisfy the necessary criteria for \Cref{l:boundingfibres}.
	
	For every $q \in [1,\infty]$ with $q \neq 2$, the normed space $\ell_q^d$ has exactly $2^d\binom{d}{2}$ linear isometries.
	Hence,
	each fibre $F^{-1}(F(p,s_n))$ for $n \in \mathbb{N}$ contains exactly $2^d \binom{d}{2}$ points, and $(G,p)$ is globally rigid in $\ell_\infty^d$ if and only if the fibre $F^{-1}(F(p,0))$ contains at most $2^d\binom{d}{2}$ points.
	By \Cref{l:boundingfibres}, the fibre $F^{-1}(F(p,0))$ contains at most $2^d\binom{d}{2}$ points, which concludes the result.
\end{proof}

\subsection{Applications of \texorpdfstring{\Cref{t:lpapprox}}{approximation theorem}}

When \Cref{t:lpapprox} is combined with \Cref{t:shin},
we obtain a powerful method for determining when generic rigid frameworks in the $\ell_\infty$ plane are globally rigid.

\begin{proof}[Proof of \Cref{t:linfplane}]
	If $G$ is $\mathcal{M}(2,2)$-connected then $(G,p)$ is globally rigid in $\ell_\infty^2$ by \Cref{t:shin} and \Cref{t:lpapprox}.
	If $(G,p)$ is globally rigid in $\ell_\infty^2$,
	then $G$ is $\mathcal{M}(2,2)$-connected by \Cref{t:globness}.
\end{proof}

We can also use \Cref{t:lpapprox} to show that global rigidity is preserved when projecting from $\ell_\infty^{d+1}$ to $\ell_\infty^d$.
We begin with the following useful lemma.

\begin{lemma}\label{lem:approxlpinfrigid}
	Let $(G,p)$ be an infinitesimally rigid framework in $\ell_\infty^{d}$.
	Then $(G,p)$ is infinitesimally rigid in $\ell_q^{d}$ for sufficiently large $q$.
\end{lemma}

\begin{proof}
	Let $F,f_{G,q}$ be the functions and $X,Y$ the sets described in \Cref{t:lpapprox} that are defined with respect to $(G,p)$.
	Hence, $dF(p,0)$ is injective.
	As $F$ is continuously rigid on $Y \times (-1,1)$,
	it follows that for sufficiently small $t>0$ we have $dF(p,t)$ is injective also.
	We now observe that for sufficiently large $q>1$ we have
	\begin{equation*}
		0 \leq \dim \ker df_{G,q}(p) -d \leq \dim \ker F(p,1/q) =0. \qedhere
	\end{equation*}
\end{proof}

\begin{theorem}\label{t:project}
	Let $(G,p)$ be a generic rigid framework in $\ell_\infty^{d+1}$,
	and let $(G,p^*)$ be the framework in $\ell_\infty^{d}$ formed by deleting the $(d+1)$-th coordinate from each vertex.
	Then $(G,p^*)$ is globally rigid in $\ell_\infty^{d}$.
\end{theorem}

\begin{proof}
	As $(G,p)$ is generic (and hence well-positioned) and rigid in $\ell_\infty^{d+1}$,
	it is infinitesimally rigid in $\ell_\infty^{d+1}$ by \Cref{t:dk}.
	By \Cref{t:colour},
	every monochromatic subgraph of $(G,p)$ is connected.
	We first note that $\phi_{p^*}$ is formed from $\phi_p$ by replacing each edge with colour $(0,\ldots,0,\pm 1)$ with another colour.
	Hence every monochromatic subgraph of $(G,p^*)$ is connected,
	and so $(G,p^*)$ is infinitesimally rigid in $\ell_\infty^{d}$ by \Cref{t:colour}.
	
	By \Cref{lem:approxlpinfrigid},
	there exists an even integer $N$ such that $(G,p)$ is infinitesimally rigid in $\ell_q^d$ for each even integer $q \geq N$.
	As $(G,p^*)$ is generic in $\mathbb{R}^{d}$,
	the framework $(G,p^*)$ is globally rigid in $\ell_q^{d}$ for $q \geq N$ by \Cref{t:shin2}.
	We now apply \Cref{t:lpapprox} to $(G,p^*)$.
\end{proof}

Using \Cref{p:k2d} and \Cref{t:project}, it is immediate that every sufficiently large complete graph is globally rigid in $\ell_\infty^d$.

\begin{corollary}\label{cor:k2d2}
	For every $d \geq 1$ and every $n \geq 2d+2$,
	the complete graph $K_{n}$ has a generic globally rigid realisation in $\ell_\infty^d$.
\end{corollary}

As a consequence of \Cref{t:globness}, any complete graph with $1 < n \leq 2d$ vertices has no generic globally rigid realisations in $\ell_\infty^d$. 
It is not clear what happens whether or not $K_n$ is globally rigid in $\ell_\infty^d$ when $n=2d+1$. It follows from \Cref{cor:linfplane} that global rigidity holds when $d=2$; i.e., $K_5$ has generic globally rigid realisations in $\ell_\infty^2$. 
The author conjectures that this is also true for all dimensions.

\begin{conjecture}\label{con:k2d2}
	The complete graph $K_{2d+1}$ has a generic globally rigid realisation in $\ell_\infty^d$.
\end{conjecture}

\section{Sufficient conditions for global rigidity}\label{sec:suff}

Unfortunately, our previous results (outside of dimension 2) can be rather difficult to work with: \Cref{thm:compute} requires large amounts of computations, and \Cref{t:lpapprox} requires a strong understanding of how global rigidity behaves in infinitely-many $\ell_q$ spaces.
Our aim in this section is to obtain a more workable sufficient condition for global rigidity in higher dimensions.
In particular, we prove \Cref{t:main2},
which describes an easy-to-check sufficient condition for a generic framework to be global rigidity in $\ell_\infty^d$.

\subsection{Strong directed colourings}

We begin with the following definition.

\begin{definition}\label{def:strong}
	Let $\phi$ be a directed colouring of $G$ with respect to the polyhedral normed space $(\mathbb{R}^d,\|\cdot\|_\mathcal{P})$.
	We say $\phi$ is \emph{strong} if for every $\phi' \in \Phi_G$,
	the inclusion $ \col M(G,\phi) \subseteq \col M(G,\phi')$ holds if and only if $\phi'$ is isometric to $\phi$.
\end{definition}

Our main interest in strong directed colourings is their application for determining global rigidity.

\begin{theorem}\label{t:main1}
    Let $(G,p)$ be a generic framework in $(\mathbb{R}^d,\|\cdot\|_\mathcal{P})$.
    If $\phi_p$ is a strong directed colouring,
    then $(G,p)$ is globally rigid in $(\mathbb{R}^d,\|\cdot\|_\mathcal{P})$.
\end{theorem}

\begin{proof}
    Choose any $\phi \in \Phi_G$ where $f_G(p) \in \col M(G,\phi)$.
    By \Cref{l:generic},
    $\col M(G,\phi_p) \subset \col M(G,\phi)$,
    and, since $\phi_p$ is strong, $\phi$ is isometric to $\phi_p$.
    Hence by \Cref{thm:compute},
    $(G,p)$ is globally rigid.
\end{proof}

The existence of a strong directed colouring for a graph does not, however, guarantee the existence of a framework with a strong directed colouring.
It does, however, guarantee that a framework satisfies Hendrickson's condition.

\begin{proposition}\label{p:strcol2}
    Let $(G,p)$ be a well-positioned and rigid framework in $(\mathbb{R}^d,\|\cdot\|_\mathcal{P})$.
    If $\phi_p$ is a strong directed colouring, then $(G,p)$ is redundantly rigid and $G$ is $\mathcal{M}(d,d)$-connected.
\end{proposition}

\begin{proof}
    ($G$ is $\mathcal{M}(d,d)$-connected):
    Perturb $(G,p)$ to a sufficiently close generic realisation $(G,\tilde{p})$ where $\phi_p=\phi_{\tilde{p}}$.
    By \Cref{t:main1},
    $(G,\tilde{p})$ is globally rigid.
    Hence, $G$ is $\mathcal{M}(d,d)$-connected by \Cref{t:globness}.
    
    ($(G,p)$ is redundantly rigid):
    Suppose for contradiction that $\phi_p$ is strong, $(G,p)$ is rigid, and there exists an edge $vw \in E$ such that $(G-vw,p)$ is not rigid.
    As $\rank df_{G-vw}(p) = \rank df_G(p)-1$ and $ df_G(p) =M(G,\phi_p)$,
    it follows that removing the row of $M(G,\phi_p)$ corresponding to $vw$ will change the rank.
    Hence for each $\sigma \in \ker M(G,\phi_p)^T$ we have $\sigma(vw) = 0$.
    Define $\phi$ to be the directed colouring where $\phi(vw) = -\phi_p(vw)$ and $\phi(e) = \phi_p(e)$ for every edge $e \neq vw$.
    As linear dependencies of the rows of $M(G,\phi)$ and $M(G,\phi_p)$ are identical,
    we have $\col M(G,\phi) = \col M(G,\phi_p)$.
    Since $\phi_p$ is strong, there exists a linear isometry $T$ so that $\phi = T \circ \phi_p$.
    Pick any vertex $u \notin \{v,w\}$ and let $F$ be all the edges with $u$ as an end.
    The set of vectors $\{ \phi_p(e) : e \in F \}$ must span $\mathbb{R}^d$, or else $(G,p)$ would be flexible.
    Since $T \circ \phi_p(e) = \phi(e)$ for each $e \in F$, it follows that $T = I_d$ and $\phi = \phi_p$.
    However this now leads to a contradiction,
    since $\phi_p(vw) \neq \phi(vw)$ and $\mathbf{0} \notin \mathcal{F}$.
\end{proof}

It follows from \Cref{cor:linfplane} and \Cref{p:strcol2} that the converse of \Cref{t:main1} is false for $\ell_\infty^2$;
for example, every $\mathcal{M}(2,2)$-circuit has a well-positioned globally rigid realisation in $\ell_\infty^2$, but has no redundantly rigid realisations in $\ell_\infty^2$ due to only having $2|V|-1$ edges.
Additionally assuming $|E| \geq d|V|$ does not help matters either; 
for example, $K_5$ has well-positioned globally rigid realisations in $\ell_\infty^2$, but has no redundantly rigid realisations in $\ell_\infty^2$ (\Cref{ex:k5}).

\subsection{Strongly directed colourings in \texorpdfstring{$\ell_\infty$}{L-infinity} normed spaces}

For this section we fix our normed space to be $\ell_\infty^d = (\mathbb{R}^d,\| \cdot\|_\infty)$, and we use the notation described in \Cref{subsec:linf} throughout. 

Our aim now is to prove \Cref{t:main2}.
To do so, we require the following key result.

\begin{lemma}\label{l:linfstrongcol}
    Let $\phi$ be a directed colouring of a graph $G=(V,E)$ where every monochromatic subgraph of $(G,\phi)$ is 2-connected.
    Then $\phi$ is a strong directed colouring.
\end{lemma}

Before we prove \Cref{l:linfstrongcol}, we set out the following notation.

Let $G=(V,E)$ be a graph with a fixed vertex ordering and choose any directed colouring $\phi \in \Phi_G$.
With this we define the $|E_i(\phi)| \times |V|$ oriented incidence matrix $I(G_i(\phi))$ by setting every entry $(e,u) \in E \times V$ ($e = vw$ with $v <w$) to be
\begin{align*}
	I(G_i(\phi))_{e,u} =
	\begin{cases}
		1 &\text{if } u = v  \text{ and } \sgn(\phi(e)) = 1, \text{ or } u = w \text{ and } \sgn(\phi(e)) = -1,\\
		-1 &\text{if } u = v  \text{ and } \sgn(\phi(e)) = -1, \text{ or } u = w \text{ and } \sgn(\phi(e)) = 1,\\
		0 &\text{otherwise.}
	\end{cases}
\end{align*}
Throughout the remainder of the section we assume that the rows and columns of each matrix $M(G,\phi)$ have been reordered so that
\begin{align*}
	M(G,\phi) = I(G_1(\phi)) \oplus \cdots \oplus I(G_d(\phi)) =
	\begin{bmatrix}
		I(G_1(\phi)) & & \\
		& \ddots & \\
		& &  I(G_d(\phi))
	\end{bmatrix},
\end{align*}
where every off-diagonal block matrix is an all-zeroes matrix.

Let $C = \{v_1v_2, v_2v_3 ,\ldots , v_{k-1} v_k, v_k v_1\}$ be a cycle of $G$.
We now define the vector $\delta_{C,\phi} \in \mathbb{R}^E$ as follows:
\begin{align*}
    \delta_{C,\phi}(e)
    := 
    \begin{cases}
        \sgn(\phi)(e) &\text{if } e = v_i v_{i+1}, ~ v_i < v_{i+1} \text{ (with $v_{k+1} = v_1$)},  \\
        -\sgn(\phi)(e) &\text{if } e = v_i v_{i+1}, ~ v_i > v_{i+1} \text{ (with $v_{k+1} = v_1$)},  \\
        0 &\text{otherwise}.
    \end{cases}
\end{align*}
%
%\begin{lemma}\label{l:cycles}
%    Given a graph $G=(V,E)$, let $\phi \in \Phi_G$ and let $C$ be a cycle of edges in $G$.
%    Then $\delta_{C,\phi} \in \ker M(G,\phi)^T$ if and only if $C$ is a monochromatic cycle of $(G,\phi)$,
%    i.e., $|\phi|(C) = \{i\}$ for some $i \in \{1,\ldots,d\}$.
%\end{lemma}
%
%\begin{proof}
%    If $|\phi|(C) = \{i\}$ then $\delta_{C,\phi}|_{E_i(\phi)} \in \ker I(G_i(\phi))^T$,
%    hence $\delta_{C,\phi} \in \ker M(G,\phi)^T$.
%    If $|\phi|(C)$ contains more than one element then each set $C \cap E_j$ contains no cycles (i.e., corresponds to a forest),
%    and so each vector $\delta_{C,\phi}|_{E_j}$ (assuming $C \cap E_j \neq \emptyset$) is not contained in $\ker I(G_i(\phi))^T$.
%\end{proof}
%
For each $i \in \{1,\ldots,d\}$,
we set $\mathcal{C}_i(\phi)$ to be the set of all cycles contain in $G_i(\phi)$.

\begin{lemma}\label{l:cyclespace}
    Let $G$ be a graph and $\phi \in \Phi_G$.
    Then the set $\bigcup_{i=1}^d \{\delta_{C,\phi} : C \in \mathcal{C}_i(\phi)\}$ forms a basis of $\ker M(G,\phi)^T$.
\end{lemma}

\begin{proof}
	Let $\phi \in \Phi_G$.    
    Then each set $\{\delta_{C,\phi} : C \in \mathcal{C}_i(\phi)\}$ forms a basis of $I(G_i(\phi))$.
    The result now follows from the structure of $M(G,\phi)$.
\end{proof}

\begin{lemma}\label{l:equivtypes}
    Let $\phi,\lambda \in \Phi_G$.
    Then the following properties are equivalent:
    \begin{enumerate}
        \item \label{l:equivtypes1} $\col M(G,\phi) = \col M(G,\lambda)$.
        \item \label{l:equivtypes2} $\ker M(G,\phi)^T = \ker M(G,\lambda)^T$.
        \item \label{l:equivtypes3} $\col M(G,\phi) \subseteq \col M(G,\lambda)$ and $\delta_{C,\phi} \in \ker M(G,\lambda)^T$ for every monochromatic cycle $C$ of $(G,\phi)$.
    \end{enumerate}
\end{lemma}

\begin{proof}
    \ref{l:equivtypes1} and \ref{l:equivtypes2} are equivalent as $\ker M(G,\psi)^T$ is the space of annihilators of $\col M(G,\psi)$ for any $\psi \in \Phi_G$.
    It is immediate that both \ref{l:equivtypes1} and \ref{l:equivtypes2} combined imply \ref{l:equivtypes3}.
    Suppose that \ref{l:equivtypes3} holds.
    By \Cref{l:cyclespace},
    $\ker M(G,\phi)^T \subseteq \ker M(G,\lambda)^T$,
    and so $\col M(G,\phi) \supseteq \col M(G,\lambda)$.
    \ref{l:equivtypes1} now follows from our assumption that $\col M(G,\phi) \subseteq \col M(G,\lambda)$.
\end{proof}

For the next lemma, we use the following observation liberally:
if $I(G,\sigma)$ is an oriented incidence matrix and $z \in \ker I(G,\sigma)^T$,
then every connected component of the graph with vertex set $V$ and edge set $\{e \in E : z(e) \neq 0\}$ is either an isolated vertex or 2-edge-connected.

\begin{lemma}\label{l:cycles2}
    Given a graph $G=(V,E)$, let $\phi ,\lambda \in \Phi_G$ and let $C$ be a cycle of edges in $G$.
    If $\delta_{C,\phi} \in \ker M(G,\lambda)^T$ then $C$ is a monochromatic cycle of $(G,\lambda)$ and $\delta_{C,\phi} = \pm \delta_{C,\lambda}$.
\end{lemma}

\begin{proof}
	As $M(G,\lambda)$ is the direct product of $I(G_1(\lambda)), \ldots, I(G_d(\lambda))$,
	there exists unique vectors $z_i \in \ker I(G_i(\lambda))$ such that $\delta_{C,\phi} = \sum_{i=1}^d z_i$ and $C$ is the disjoint union of the sets $C_i := \{ e \in E : z_i(e)\}$.
	We note now that each set $C_i$ is 2-edge-connected and $C_i \cap C_j = \emptyset$ for any $i \neq j$. 
	Since each set $C_i$ is contained in $C$,
	it follows that there exists a unique $j \in [d]$ so that $\delta_{C,\phi}=z_j$ and $z_i=\mathbf{0}$ for each $i \neq j$.
	Hence, $C$ is a monochromatic cycle of $(G,\lambda)$; specifically, $C$ has colour $j$ with respect to $\lambda$.
	
	Let $z := (\delta_{C,\lambda} - \delta_{C,\phi})/2 \in \{-1,0,1\}^E$.
	Then the edge set $C_z := \{ e \in E : z(e) \neq 0\}$ is contained in $C$.
	Since $z \in \ker I(G_j(\lambda))$,
	each connected component of $(V,C_z)$ is either an isolated vertex or 2-edge-connected.
	As cycles have no proper 2-edge-connected subgraphs,
	either $C_z= C$ and $\delta_{C,\lambda} = -\delta_{C,\phi}$ or $C_z= \emptyset$ and $\delta_{C,\lambda} = \delta_{C,\phi}$.
\end{proof}

\begin{lemma}\label{l:samemonosub}
    Let $G=(V,E)$ be a graph with directed colourings $\phi,\lambda$.
    Suppose that every monochromatic subgraph of $(G,\phi)$ is 2-connected.
    If $\col M(G,\phi) \subset \col M(G,\lambda)$,
    then the following holds.
    \begin{enumerate}
    	\item\label{l:samemonosub.0} $\col M(G,\phi) = \col M(G,\lambda)$.
        \item\label{l:samemonosub.1} There exists a permutation $\sigma$ of $\{1,\ldots,d\}$ such that $G_{\sigma(i)}(\lambda) = G_{i}(\phi)$.
        \item\label{l:samemonosub.2} If two edges $e,f\in E$ are contained in a monochromatic cycle of $(G,\phi)$,
        then
        \begin{equation*}
            \sgn(\phi)(e) \sgn(\phi)(f) = \sgn(\lambda)(e) \sgn(\lambda)(f).
        \end{equation*}
    \end{enumerate}
\end{lemma}

\begin{proof}
    \ref{l:samemonosub.0}: 
    We note that $\rank M(G,\psi) = \dim \col M(G,\psi) \leq d|V| - d$ for any $\psi \in \Phi_G$.
    By \Cref{t:colour,t:dk},
    $\rank M(G,\phi) = d|V|-d$.
    \ref{l:samemonosub.0} now follows as $\col M(G,\phi) \subset \col M(G,\lambda)$.
    
    \ref{l:samemonosub.1}: 
    By combining \ref{l:samemonosub.0} and \Cref{l:equivtypes},
    we have that $\delta_{C,\phi} \in \ker M(G,\lambda)^T$ for every monochromatic cycle $C$ of $(G,\phi)$.
    Hence, by \Cref{l:cycles2},
    every monochromatic cycle of $(G,\phi)$ is a monochromatic cycle of $(G,\lambda)$.
    Using \ref{l:samemonosub.0}, we have $\col M(G,\lambda) \subset \col M(G,\phi)$ also.
    Hence, by applying \Cref{l:equivtypes,l:cycles2} again with $\phi,\lambda$ switched,
    we see that the monochromatic cycles of $(G,\phi)$ are exactly the monochromatic cycles of $(G,\lambda)$.
    Choose any two edges $e,f$ of $G$ where $|\phi|(e) = |\phi|(f) = i$.
    As every monochromatic subgraph of $(G,\phi)$ is 2-connected,
    there exists a monochromatic cycle $C$ of $(G,\phi)$ containing both $e$ and $f$.
    As $C$ is also a monochromatic cycle of $(G,\lambda)$,
    $|\lambda|(e) = |\lambda|(f)$ also.
    Hence, there exists a permutation $\sigma$ of $\{1,\ldots,d\}$ such that $G_{\sigma(i)}(\lambda) = G_{i}(\phi)$.
        
    \ref{l:samemonosub.2}: 
    This follows immediately from \Cref{l:cycles2}.
\end{proof}

% For each monochromatic subgraph $G_i^\phi$ of $(G,\phi)$,
% define the directed incidence matrix $I(G_i^\phi)$ where for each edge $e=vw $ with $v<w$ and each vertex $u$ we have
% \begin{align*}
%     I(G_i^\phi)_{e,u} = 
%     \begin{cases}
%         1  &\text{if } u =v\\
%         -1 &\text{if } u = w\\
%         0 &\text{otherwise.} \\
%     \end{cases}
% \end{align*}
% With this we have that $\col M(G,\phi) = \bigoplus_{i=1}^d \col I(G_i^\phi)$.
% This can be seen as the rows and columns of the matrix $M(G,\phi)$ can be rearranged to form the matrix
% \begin{align*}
%     \begin{bmatrix}
%         I(G_1^\phi) & & 0\\
%         & \ddots & \\
%         0 & & I(G_d^\phi)
%     \end{bmatrix}.
% \end{align*}

We are now ready to prove our key lemma.

\begin{proof}[Proof of \Cref{l:linfstrongcol}]
    Choose any $\lambda \in \Phi_G$ where $\col M(G,\phi) \subseteq \col M(G,\lambda)$.
    By \Cref{l:samemonosub}\ref{l:samemonosub.1},
    we may suppose (by applying rotational isometries to $\lambda$ if necessary) that $G_i(\lambda) = G_i(\phi)$ for each $i \in \{1,\ldots,d\}$.
    For each $i \in \{1,\ldots,d\}$,
    choose an edge $e_i \in E_i(\phi)$.
    By applying reflections to $\lambda$ where needed, we may assume $\sgn(\phi)(e_i) = \sgn(\lambda)(e_i)$ for each $i \in \{1,\ldots,d\}$.
    Fix $i \in \{1,\ldots,d\}$ and choose any edge $e \in E_i(\lambda) = E_i(\phi)$.
    As $G_i(\lambda)$ is 2-connected, there exists a monochromatic cycle containing $e$ and $e_i$.
    By \Cref{l:samemonosub}\ref{l:samemonosub.2},
    we have $\sgn(\phi)(e) = \sgn(\lambda)(e)$.
    It now follows that $\lambda = \phi$, and thus $\phi$ is a strong directed colouring.
\end{proof}

With this, we are now ready to prove \Cref{t:main2}.

\begin{proof}[Proof of \Cref{t:main2}]
    The result follows immediately from \Cref{l:linfstrongcol} and \Cref{t:main1}.
\end{proof}

An immediate corollary from \Cref{t:main2} and \Cref{p:singleexample} is the following.

\begin{corollary}\label{cor:main2}
	Let $(G,p)$ be a well-positioned framework in $\mathbb{R}^d$.
    Suppose that every monochromatic subgraph of $(G,\phi_p)$ is 2-connected.
    Then there exists a generic globally rigid realisation of $G$ in $\ell_{\infty}^d$.
\end{corollary}

\begin{example}\label{ex:oct}
	Fix $G=(V,E)$ to be the 1-skeleton of the octahedron, with vertex set $V = \{v_{-3},v_{-2},v_{-1},v_1,v_2.v_3\}$ and edge set $E = \{ v_i v_j : |i| \neq |j| \}$.
	Let $p$ be the well-positioned realisation of $G$ in $\ell_\infty^2$ where
	\begin{gather*}
		p(v_{-1}) = (0,0.9), \qquad p(v_{-2}) = (1,0), \qquad p(v_{-3}) = (0,0) \\
		p(v_1) = (1, 0.9), \qquad p(v_2) = (1,1.8), \qquad p(v_3) = (0,1.8).
	\end{gather*}
	As can be seen in \Cref{fig:oct},
	the monochromatic subgraphs of $(G,\phi_p)$ are both 2-connected,
	and so $(G,p)$ is redundantly rigid in $\ell_\infty^2$.
	(In fact, as a consequence of \Cref{ex:k5}, $G$ is the smallest graph with a redundantly rigid realisation in $\ell_\infty^2$.)
	By \Cref{cor:main2}, $G$ has a generic globally rigid realisation in $\ell_\infty^2$.
\end{example}

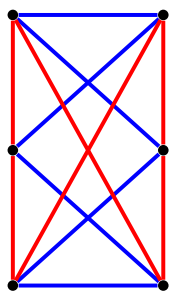
\begin{figure}[ht]
\begin{center}
\begin{tikzpicture}[scale=2]
		\node[vertex] (-v3) at (0,0) {};
		\node[vertex] (-v2) at (1,0) {};
		\node[vertex] (-v1) at (0,0.9) {};
		\node[vertex] (v3) at (0,1.8) {};
		\node[vertex] (v2) at (1,1.8) {};
		\node[vertex] (v1) at (1,0.9) {};

		\draw[edge, red] (v1)edge(v2);
		\draw[edge, blue] (v1)edge(v3);
		\draw[edge, red] (v1)edge(-v2);
		\draw[edge, blue] (v1)edge(-v3);
		
		\draw[edge, blue] (-v1)edge(v2);
		\draw[edge, red] (-v1)edge(v3);
		\draw[edge, blue] (-v1)edge(-v2);
		\draw[edge, red] (-v1)edge(-v3);

		\draw[edge, blue] (v2)edge(v3);
		\draw[edge, red] (v2)edge(-v3);

		\draw[edge, red] (-v2)edge(v3);
		\draw[edge, blue] (-v2)edge(-v3);
	\end{tikzpicture}
	\end{center}
	\caption{The monochromatic subgraphs of the framework $(G,p)$ given in \Cref{ex:oct}: blue represents an edge with corresponding face $(\pm1,0)$, red represents an edge with corresponding face $(0,\pm1)$.}
	\label{fig:oct}
\end{figure}

\subsection*{Acknowledgement}
The author was supported by the Heilbronn Institute for Mathematical Research.
%I would like to thank Shin-ichi Tanigawa for his helpful discussion surrounding \Cref{subsec:farkas}.

\bibliographystyle{plainurl}
\bibliography{ref}

\end{document}